\documentclass{amsart}
\usepackage{amssymb}
\usepackage{amsfonts}

\newtheorem{theorem}{Theorem}
\theoremstyle{plain}

\newtheorem{corollary}{Corollary}

\newtheorem{lemma}{Lemma}

\newtheorem{remark}{Remark}

\numberwithin{equation}{section}
\input{tcilatex}

\begin{document}
\title[Failure of energy reversal]{A note on failure of energy reversal for
classical fractional singular integrals}

\begin{abstract}
For $0\leq \alpha <n$ we demonstrate the failure of energy reversal for the
vector of $\alpha $-fractional Riesz transforms, and more generally for the
vector of all $\alpha $-fractional convolution singular integrals having a
kernel with vanishing integral on every great circle of the sphere.
\end{abstract}

\author[E.T. Sawyer]{Eric T. Sawyer}
\address{ Department of Mathematics \& Statistics, McMaster University, 1280
Main Street West, Hamilton, Ontario, Canada L8S 4K1 }
\email{sawyer@mcmaster.ca}
\thanks{Research supported in part by NSERC}
\author[C.-Y. Shen]{Chun-Yen Shen}
\address{ Department of Mathematics \\
National Central University \\
Chungli, Taiwan 32054}
\email{chunyshen@gmail.com}
\thanks{Research supported in part by the NSC, through grant
NSC102-2115-M-008-015-MY2}
\author[I. Uriarte-Tuero]{Ignacio Uriarte-Tuero}
\address{ Department of Mathematics \\
Michigan State University \\
East Lansing MI }
\email{ignacio@math.msu.edu}
\thanks{ I. Uriarte-Tuero has been partially supported by grants DMS-1056965
(US NSF), MTM2010-16232, MTM2009-14694-C02-01 (Spain), and a Sloan
Foundation Fellowship. }
\maketitle
\tableofcontents

\section{Introduction}

To set notation we recall a special case of Theorem 1 from our paper \cite%
{SaShUr}, using notation from that paper.

\begin{theorem}
\label{special}Suppose that $\sigma $ and $\omega $ are locally finite
positive Borel measures in $\mathbb{R}^{n}$ with no common point masses, and
assume the finiteness of the $\alpha $\emph{-energy condition} constant%
\begin{eqnarray*}
\left( \mathcal{E}_{\alpha }\right) ^{2} &\equiv &\sup_{\substack{ Q=\dot{%
\cup}Q_{r} \\ Q,Q_{r}\in \mathcal{D}^{n}}}\frac{1}{\left\vert I\right\vert
_{\sigma }}\sum_{r=1}^{\infty }\sum_{J\in \mathcal{M}_{\mathbf{r}-\limfunc{%
deep}}\left( Q_{r}\right) }\left( \frac{\mathrm{P}^{\alpha }\left( J,\mathbf{%
1}_{Q}\sigma \right) }{\left\vert J\right\vert ^{\frac{1}{n}}}\right)
^{2}\left\Vert \mathsf{P}_{J}^{\limfunc{subgood},\omega }\mathbf{x}%
\right\Vert _{L^{2}\left( \omega \right) }^{2} \\
&&+\sup_{\ell \geq 0}\frac{1}{\left\vert I\right\vert _{\sigma }}\sum_{J\in 
\mathcal{M}_{\mathbf{r}-\limfunc{deep}}^{\ell }\left( Q\right) }\left( \frac{%
\mathrm{P}^{\alpha }\left( J,\mathbf{1}_{Q}\sigma \right) }{\left\vert
J\right\vert ^{\frac{1}{n}}}\right) ^{2}\left\Vert \mathsf{P}_{J}^{\limfunc{%
subgood},\omega }\mathbf{x}\right\Vert _{L^{2}\left( \omega \right) }^{2},
\end{eqnarray*}%
and its dual, uniformly over all dyadic grids $\mathcal{D}^{n}$, and where
the goodness parameters $\mathbf{r}$ and $\varepsilon $ implicit in the
definition of $\mathcal{M}_{\limfunc{deep}}\left( K\right) $ are fixed
sufficiently large and small respectively depending on $n$ and $\alpha $.
Let $\mathbf{T}^{\alpha }$ be a standard strongly elliptic $\alpha $%
-fractional Calder\'{o}n-Zygmund operator in Euclidean space $\mathbb{R}^{n}$%
. Then $\mathbf{T}^{\alpha }$ is bounded from $L^{2}\left( \sigma \right) $
to $L^{2}\left( \omega \right) $ \emph{if and only if} the $\mathcal{A}%
_{2}^{\alpha }$ condition%
\begin{equation}
\mathcal{A}_{2}^{\alpha }\equiv \sup_{Q\in \mathcal{Q}^{n}}\mathcal{P}%
^{\alpha }\left( Q,\sigma \right) \frac{\left\vert Q\right\vert _{\omega }}{%
\left\vert Q\right\vert ^{1-\frac{\alpha }{n}}}<\infty   \label{A2}
\end{equation}%
and its dual hold, the cube testing conditions%
\begin{equation}
\int_{Q}\left\vert \mathbf{T}^{\alpha }\left( \mathbf{1}_{Q}\sigma \right)
\right\vert ^{2}\omega \leq \mathfrak{T}_{T^{\alpha }}^{2}\int_{Q}d\sigma 
\text{ and }\int_{Q}\left\vert \left( \mathbf{T}^{\alpha }\right) ^{\ast
}\left( \mathbf{1}_{Q}\omega \right) \right\vert ^{2}\sigma \leq \mathfrak{T}%
_{T^{\alpha }}^{2}\int_{Q}d\omega ,  \label{testing}
\end{equation}%
hold for all cubes $Q$ in $\mathbb{R}^{n}$, and the weak boundedness
property for $\mathbf{T}^{\alpha }$ holds:%
\begin{eqnarray*}
&&\left\vert \int_{Q}\mathbf{T}^{\alpha }\left( 1_{Q^{\prime }}\sigma
\right) d\omega \right\vert \leq \mathcal{WBP}_{\mathbf{T}^{\alpha }}\sqrt{%
\left\vert Q\right\vert _{\omega }\left\vert Q^{\prime }\right\vert _{\sigma
}}, \\
&&\ \ \ \ \ \text{for all cubes }Q,Q^{\prime }\text{ with }\frac{1}{C}\leq 
\frac{\left\vert Q\right\vert ^{\frac{1}{n}}}{\left\vert Q^{\prime
}\right\vert ^{\frac{1}{n}}}\leq C, \\
&&\ \ \ \ \ \text{and either }Q\subset 3Q^{\prime }\setminus Q^{\prime }%
\text{ or }Q^{\prime }\subset 3Q\setminus Q.
\end{eqnarray*}
\end{theorem}

In \cite{SaShUr3} we used Theorem \ref{special} to prove the $T1$ theorem
for the vector of Riesz transforms in $\mathbb{R}^{n}$ in the special case
when one of the measures $\sigma ,\omega $ is supported on a line in $%
\mathbb{R}^{n}$. The key to that proof was proving control of the above
energy constants $\mathcal{E}_{\alpha }$ and $\mathcal{E}_{\alpha }^{\ast }$
in terms of the constants in the hypotheses (\ref{A2}) and (\ref{testing}).
A number of attempts have been made by us and others (see e.g. earlier
versions of \cite{SaShUr} and \cite{LaWi}) to prove such control of various
different energy conditions by invoking an \emph{energy reversal} for the
Riesz transforms and similar operators - see (\ref{will fail}) below - but
all of these attempts have been met with failure. The purpose of this short
note is to show first that \emph{energy reversal} is false, not only for the
vector of $\alpha $-fractional Riesz transforms in the plane when $0\leq
\alpha <2$, but also for the vectors of classical $\alpha $-fractional
singular integrals in the plane, 
\begin{eqnarray*}
\mathbf{T}_{M}^{\alpha } &\equiv &\left\{ T_{\Omega }:\Omega \in \mathcal{P}%
_{M}\right\} , \\
\mathcal{P}_{M} &\equiv &\left\{ \cos n\theta ,\sin n\theta \right\}
_{n=1}^{M}\ ,
\end{eqnarray*}%
where $T_{\Omega }^{\alpha }$ has convolution kernel $\frac{\Omega \left( 
\frac{x}{\left\vert x\right\vert }\right) }{\left\vert x\right\vert
^{2-\alpha }}=\frac{\Omega \left( \theta \right) }{\left\vert x\right\vert
^{2-\alpha }}$ and $0\leq \alpha <2$. The linear space $\mathcal{L}_{M}$ of
trigonometric polynomials with vanishing mean and degree at most $M$ is
spanned by the monomials $\mathcal{P}_{M}$, and so we also obtain the
failure of energy reversal for the infinite vector $\mathbf{T}_{M}^{\alpha
}\equiv \left\{ T_{\Omega }:\Omega \in \mathcal{L}_{M}\right\} $. A standard
limiting argument applied to the proof below extends this failure to all
sufficiently smooth $\Omega \left( \theta \right) $ with vanishing mean on
the circle. Finally, we embed an analogue of the planar measure constructed
below into Euclidean space $\mathbb{R}^{n}$ in order to obtain the failure
of energy reversal for \emph{any} vector of classical convolution Calder\'{o}%
n-Zygmund operators with odd kernel in $\mathbb{R}^{n}$ - and more generally
for kernels $\frac{\Omega \left( x^{\prime }\right) }{\left\vert
x\right\vert ^{n-\alpha }}$ where $\Omega $ has vanishing integral on every
great circle in the sphere $\mathbb{S}^{n-1}$. A key to our proof is the
positivity of the determinants $\det \left[ \frac{\Gamma \left( z\right) ^{2}%
}{\Gamma \left( z-\left\vert i-j\right\vert \right) \Gamma \left(
z+\left\vert i-j\right\vert \right) }\right] _{i,j=1}^{n}$ for all $n\geq 1$%
. See also \cite{LaWi} for related results regarding fractional Riesz
transforms in higher dimensions. We thank Michael Lacey for pointing out to
us that the $1$-fractional Riesz transform $\mathbf{R}^{1}\sigma \left(
z\right) =\int_{\mathbb{T}}\frac{z-\xi }{\left\vert z-\xi \right\vert ^{2}}%
d\sigma \left( \xi \right) $ of the unit circle measure $\sigma $ vanishes
identically for $z$ inside the unit disk. Indeed, $\mathbf{R}^{1}\sigma $ is
the gradient of the planar Newtonian potential $\mathbf{N}\sigma \left(
z\right) =\int_{\mathbb{T}}\ln \left\vert z-\xi \right\vert d\sigma \left(
\xi \right) $, and $\mathbf{N}\sigma $ is constant inside the disk.

\section{Failure of reversal of energy}

Recall the energy $\mathsf{E}\left( J,\omega \right) $ of $\omega $ on a
cube $J$,%
\begin{equation*}
\mathsf{E}\left( J,\omega \right) ^{2}\equiv \frac{1}{\left\vert
J\right\vert _{\omega }}\frac{1}{\left\vert J\right\vert _{\omega }}%
\int_{J}\int_{J}\left\vert \frac{x-z}{\left\vert J\right\vert ^{\frac{1}{n}}}%
\right\vert ^{2}d\omega \left( x\right) d\omega \left( z\right) =2\frac{1}{%
\left\vert J\right\vert _{\omega }}\int_{J}\left\vert \frac{x-\mathbb{E}%
_{J}^{\omega }x}{\left\vert J\right\vert ^{\frac{1}{n}}}\right\vert
^{2}d\omega \left( x\right) .
\end{equation*}%
Define its associated \emph{coordinate} energies $\mathsf{E}^{j}\left(
J,\omega \right) $ by%
\begin{equation*}
\mathsf{E}^{j}\left( J,\omega \right) ^{2}\equiv \frac{1}{\left\vert
J\right\vert _{\omega }}\frac{1}{\left\vert J\right\vert _{\omega }}%
\int_{J}\int_{J}\left\vert \frac{x^{j}-z^{j}}{\left\vert J\right\vert ^{%
\frac{1}{n}}}\right\vert ^{2}d\omega \left( x\right) d\omega \left( z\right)
,\ \ \ \ \ j=1,2,...,n,
\end{equation*}%
and the rotations $\mathsf{E}_{\mathcal{R}}^{j}\left( J,\omega \right) $ of
the coordinate energies by a rotation $\mathcal{R}\in SO\left( n\right) $,
which we refer to as \emph{partial} energies,%
\begin{equation*}
\mathsf{E}_{\mathcal{R}}^{j}\left( J,\omega \right) ^{2}\equiv \frac{1}{%
\left\vert J\right\vert _{\omega }}\frac{1}{\left\vert J\right\vert _{\omega
}}\int_{J}\int_{J}\left\vert \frac{x_{\mathcal{R}}^{j}-z_{\mathcal{R}}^{j}}{%
\left\vert J\right\vert ^{\frac{1}{n}}}\right\vert ^{2}d\omega \left(
x\right) d\omega \left( z\right) ,\ \ \ \ \ j=1,2,...,n,
\end{equation*}%
where for $\mathcal{R}\in SO\left( n\right) $, $x_{\mathcal{R}}=\left( x_{%
\mathcal{R}}^{j}\right) _{j=1}^{n}=\mathcal{R}\left( x^{j}\right) _{j=1}^{n}=%
\mathcal{R}x$. Set $\mathsf{E}_{\mathcal{R}}\left( J,\omega \right)
^{2}\equiv \mathsf{E}_{\mathcal{R}}^{1}\left( J,\omega \right) ^{2}+...+%
\mathsf{E}_{\mathcal{R}}^{n}\left( J,\omega \right) ^{2}$. We have the
following elementary computations.

\begin{lemma}
\label{partial energies}For $\mathcal{R}\in SO\left( n\right) $ we have%
\begin{equation}
\mathsf{E}_{\mathcal{R}}\left( J,\omega \right) ^{2}=\mathsf{E}_{\mathcal{R}%
}^{1}\left( J,\omega \right) ^{2}+...+\mathsf{E}_{\mathcal{R}}^{n}\left(
J,\omega \right) ^{2}=\mathsf{E}\left( J,\omega \right) ^{2}.
\label{rotated partial energies}
\end{equation}%
More generally, if $\mathfrak{R}=\left\{ \mathcal{R}_{j}\right\}
_{j=1}^{n}\subset SO\left( n\right) $ is a collection of rotations such that
the matrix $M_{\mathfrak{R}}=\left[ 
\begin{array}{c}
\mathcal{R}_{1}\mathbf{e}^{1} \\ 
\vdots \\ 
\mathcal{R}_{n}\mathbf{e}^{1}%
\end{array}%
\right] $ with rows $\mathcal{R}_{\ell }\mathbf{e}^{1}$ is nonsingular, then 
\begin{equation}
\mathsf{E}\left( J,\omega \right) ^{2}\leq \frac{1}{\epsilon _{\mathfrak{R}}}%
\dsum\limits_{\ell =1}^{n}\mathsf{E}_{\mathcal{R}_{\ell }}^{1}\left(
J,\omega \right) ^{2},  \label{est partial}
\end{equation}%
where $\epsilon _{\mathfrak{R}}$ is the least eigenvalue of $M_{\mathfrak{R}%
}^{\ast }M_{\mathfrak{R}}$.
\end{lemma}

\begin{proof}
We have%
\begin{eqnarray*}
\left\vert x_{\mathcal{R}}^{1}-z_{\mathcal{R}}^{1}\right\vert
^{2}+...+\left\vert x_{\mathcal{R}}^{n}-z_{\mathcal{R}}^{n}\right\vert ^{2}
&=&\left\vert \mathcal{R}\left( x-z\right) \right\vert ^{2} \\
&=&\left\vert x-z\right\vert ^{2}=\left\vert x^{1}-z^{1}\right\vert
^{2}+...+\left\vert x^{n}-z^{n}\right\vert ^{2},
\end{eqnarray*}%
so that%
\begin{eqnarray*}
\mathsf{E}_{\mathcal{R}}\left( J,\omega \right) ^{2} &\equiv &\mathsf{E}_{%
\mathcal{R}}^{1}\left( J,\omega \right) ^{2}+...+\mathsf{E}_{\mathcal{R}%
}^{n}\left( J,\omega \right) ^{2} \\
&=&\mathsf{E}^{1}\left( J,\omega \right) ^{2}+...+\mathsf{E}^{n}\left(
J,\omega \right) ^{2}=\mathsf{E}\left( J,\omega \right) ^{2}.
\end{eqnarray*}%
More generally, if $M_{\mathfrak{R}}^{\ell }$ denotes the $\ell ^{th}$ row
of the matrix $M_{\mathfrak{R}}$, we have%
\begin{eqnarray*}
\epsilon _{\mathfrak{R}}\left\vert x-z\right\vert ^{2} &\leq &\left(
x-z\right) ^{tr}M_{\mathfrak{R}}^{\ast }M_{\mathfrak{R}}\left( x-z\right) \\
&=&\sum_{\ell =1}^{n}\left\vert \mathcal{R}_{\ell }\mathbf{e}^{1}\cdot
\left( x-z\right) \right\vert ^{2},
\end{eqnarray*}%
so that%
\begin{eqnarray*}
\epsilon _{\mathfrak{R}}\mathsf{E}\left( J,\omega \right) ^{2} &=&\left( 
\frac{1}{\left\vert J\right\vert _{\omega }\left\vert J\right\vert ^{\frac{1%
}{n}}}\right) ^{2}\int_{J}\int_{J}\epsilon _{\mathfrak{R}}\left\vert
x-z\right\vert ^{2}d\omega \left( x\right) d\omega \left( z\right) \\
&\leq &\left( \frac{1}{\left\vert J\right\vert _{\omega }\left\vert
J\right\vert ^{\frac{1}{n}}}\right) ^{2}\int_{J}\int_{J}\left\{ \sum_{\ell
=1}^{n}\left\vert \mathcal{R}_{\ell }\mathbf{e}^{1}\cdot \left( x-z\right)
\right\vert ^{2}\right\} d\omega \left( x\right) d\omega \left( z\right) \\
&=&\sum_{\ell =1}^{n}\mathsf{E}_{\mathcal{R}_{\ell }}^{1}\left( J,\omega
\right) ^{2}.
\end{eqnarray*}
\end{proof}

The point of the estimate (\ref{est partial}) is that it could hopefully be
used to help obtain a reversal of energy for a vector transform $\mathbf{T}%
^{n,\alpha }=\left\{ T_{\ell }^{n,\alpha }\right\} _{\ell =1}^{n}$, where
the convolution kernel $K_{\ell }^{n,\alpha }\left( w\right) $ of the
operator $T_{\ell }^{n,\alpha }$ has the form%
\begin{equation}
K_{\ell }^{n,\alpha }\left( w\right) =\frac{\Omega _{\ell }^{n}\left( \frac{w%
}{\left\vert w\right\vert }\right) }{\left\vert w\right\vert ^{n-\alpha }},
\label{classical form}
\end{equation}%
and where $\Omega _{\ell }^{n}$ is smooth on the sphere $\mathbb{S}^{n-1}$.
We refer to the operator $T_{\ell }^{n,\alpha }$ as an $\alpha $-fractional 
\emph{convolution} Calder\'{o}n-Zygmund operator. If in addition we require
that $\Omega _{\ell }^{n}$ has vanishing integral on the sphere $\mathbb{S}%
^{n-1}$, we refer to $T_{\ell }^{n,\alpha }$ as a \emph{classical} $\alpha $%
-fractional Calder\'{o}n-Zygmund operator.

However, we now dash this hope, at least for the most familiar singular
operators in the plane, in a spectacular way. A vector $\mathbf{T}^{\alpha
}=\left\{ T_{\ell }^{\alpha }\right\} _{\ell =1}^{N}$ of $\alpha $%
-fractional transforms in Euclidean space $\mathbb{R}^{n}$ satisfies a \emph{%
strong} reversal of $\omega $-energy on a cube $J$ if there is a positive
constant $C_{0}$ such that for all $\gamma \geq 2$ sufficiently large and
for all positive measures $\mu $ supported outside $\gamma J$, we have the
inequality%
\begin{equation}
\mathsf{E}\left( J,\omega \right) ^{2}\mathrm{P}^{\alpha }\left( J,\mu
\right) ^{2}\leq C_{0}\ \mathbb{E}_{J}^{d\omega \left( x\right) }\mathbb{E}%
_{J}^{d\omega \left( z\right) }\left\vert \mathbf{T}^{\alpha }\mu \left(
x\right) -\mathbf{T}^{\alpha }\mu \left( z\right) \right\vert ^{2}.
\label{will fail}
\end{equation}%
We show that (\ref{will fail}) is false by stating and proving a variant of
Lemma 9 in \cite{SaShUr2}.

\begin{lemma}[Failure of Reverse Energy]
\label{LRE}Suppose that$\ J$ is a square in the plane $\mathbb{R}^{2}$, $%
0\leq \alpha <2$, $\gamma >2$ and that $\mathbf{R}^{\alpha }=\left\{ R_{\ell
}^{\alpha }\right\} _{\ell =1}^{2}$ is the vector of $\alpha $-fractional
Riesz transforms in the plane $\mathbb{R}^{2}$ with kernels $K_{\ell
}^{\alpha }\left( w\right) =\frac{\Omega _{\ell }\left( \frac{w}{\left\vert
w\right\vert }\right) }{\left\vert w\right\vert ^{2-\alpha }}$ and $\Omega
_{\ell }\left( \frac{w}{\left\vert w\right\vert }\right) =\frac{w_{\ell }}{%
\left\vert w\right\vert }$. Finally suppose that $C_{0}>0$ is given. For $%
\gamma $ sufficiently large, there exists a positive measure $\mu $ on $%
\mathbb{R}^{2}$ supported outside $\gamma J$ and depending only on $\alpha $
and $\gamma $, such that the strong reversal of energy inequality (\ref{will
fail}) \textbf{fails}. Moreover, we can choose $\mu $ as above so that in
addition, for any $M\geq 1$, the strong reversal of energy inequality (\ref%
{will fail}) fails for the vector $\mathbf{T}_{M}^{\alpha }$.
\end{lemma}

As a corollary of the proof of this lemma we easily obtain an extension to
higher dimensions by simply embedding an appropriate planar measure into
Euclidean space $\mathbb{R}^{n}$.

\begin{corollary}[of the proof of Lemma \protect\ref{LRE}]
\label{LRE porism}Suppose that$\ J$ is a cube in $\mathbb{R}^{n}$, $0\leq
\alpha <n$, $\gamma >2$ and suppose that $C_{0}>0$ is given. For $\gamma $
sufficiently large, there exists a positive measure $\mu $ on $\mathbb{R}%
^{n} $ supported outside $\gamma J$ and depending only on $n$, $\alpha $ and 
$\gamma $, such that the strong reversal of energy inequality (\ref{will
fail}) \textbf{fails} for any vector $\mathbf{T}^{\alpha }=\left\{ T_{\ell
}^{\alpha }\right\} _{\ell =1}^{N}$ of $\alpha $-fractional smooth Calder%
\'{o}n-Zygmund operators in $\mathbb{R}^{n}$ with kernels $K_{\ell }^{\alpha
}\left( w\right) =\frac{\Omega _{\ell }\left( \frac{w}{\left\vert
w\right\vert }\right) }{\left\vert w\right\vert ^{n-\alpha }}$, where $%
\Omega _{\ell }$ has vanishing integral on every great circle in the sphere $%
\mathbb{S}^{n-1}$ - in particular this holds if each $K_{\ell }^{\alpha }$
is odd.
\end{corollary}

\begin{proof}[Proof of Lemma \protect\ref{LRE} for the Riesz transform vector%
]
Let $\varepsilon >0$. We let $\frac{\Omega _{\ell }\left( \frac{w}{%
\left\vert w\right\vert }\right) }{\left\vert w\right\vert ^{2-\alpha }}$ be
an arbitrary standard kernel for the moment. With $K_{\ell }^{\alpha }\left(
x,y\right) =K_{\ell }^{\alpha }\left( x-y\right) $ we have%
\begin{eqnarray*}
T_{\ell }^{\alpha }\mu \left( x\right)  &=&\int K_{\ell }^{\alpha }\left(
x-y\right) d\mu \left( y\right) =\int \frac{\Omega _{\ell }\left( x-y\right) 
}{\left\vert y-x\right\vert ^{2-\alpha }}d\mu \left( y\right)  \\
&=&\int \left\{ K_{\ell }^{\alpha }\left( c_{J}-y\right) +\left(
x-c_{J}\right) \cdot \nabla K_{\ell }^{\alpha }\left( c_{J}-y\right)
\right\} d\mu \left( y\right) +E_{\ell ,x},
\end{eqnarray*}%
and so%
\begin{eqnarray*}
&&T_{\ell }^{\alpha }\mu \left( x\right) -T_{\ell }^{\alpha }\mu \left(
z\right)  \\
&=&\int \left\{ \left( x-z\right) \cdot \nabla K_{\ell }^{\alpha }\left(
c_{J}-y\right) \right\} d\mu \left( y\right) +\left[ E_{\ell ,x}^{\alpha
}-E_{\ell ,z}^{\alpha }\right]  \\
&\equiv &\Lambda _{\ell }^{\alpha }+\left[ E_{\ell ,x}^{\alpha }-E_{\ell
,z}^{\alpha }\right] ,
\end{eqnarray*}%
where if $\gamma >2$ is sufficiently large, 
\begin{equation}
\left\vert E_{\ell ,x}^{\alpha }-E_{\ell ,z}^{\alpha }\right\vert \leq C%
\frac{1}{\gamma ^{\delta }}\frac{\mathrm{P}^{\alpha }\left( J,\mu \right) }{%
\left\vert J\right\vert ^{\frac{1}{2}}}\left\vert x-z\right\vert \leq
\varepsilon \frac{\mathrm{P}^{\alpha }\left( J,\mu \right) }{\left\vert
J\right\vert ^{\frac{1}{2}}}\left\vert x-z\right\vert .  \label{er est''}
\end{equation}

The point of this inequality (\ref{er est''}) is that it permits the
replacement of the difference $T_{\ell }^{\alpha }\mu \left( x\right)
-T_{\ell }^{\alpha }\mu \left( z\right) $ in (\ref{will fail}) by the linear
part $\Lambda _{\ell }^{\alpha }$ of the Taylor expansion of the kernel $%
K_{\ell }^{\alpha }$.

Now we make the choice%
\begin{eqnarray*}
\Omega _{\ell }\left( w\right)  &=&\Omega \left( \theta _{\ell }\left(
w\right) \right) ; \\
\theta _{\ell }\left( w\right)  &\equiv &\tan ^{-1}\frac{\left( -1\right)
^{\ell ^{\prime }}w^{\ell ^{\prime }}}{w^{\ell }},\ \ \ \ \ 1\leq \ell \leq
2,
\end{eqnarray*}%
where $w^{\ell ^{\prime }}$ denotes the coordinate variable other than $%
w^{\ell }$, i.e. $\ell +\ell ^{\prime }=3$. Thus $\theta _{1}$ is the usual
angular coordinate on the circle and $\theta _{2}=\theta _{1}+\frac{\pi }{2}$%
. We now use%
\begin{eqnarray*}
\nabla \left\vert w\right\vert ^{\alpha -2} &=&\left( \frac{\partial }{%
\partial w^{1}}\left( \left( w^{1}\right) ^{2}+\left( w^{2}\right)
^{2}\right) ^{\frac{\alpha -2}{2}},\frac{\partial }{\partial w^{2}}\left(
\left( w^{1}\right) ^{2}+\left( w^{2}\right) ^{2}\right) ^{\frac{\alpha -2}{2%
}}\right)  \\
&=&\frac{\alpha -2}{2}\left( \left( w^{1}\right) ^{2}+\left( w^{2}\right)
^{2}\right) ^{\frac{\alpha -2}{2}-1}2w \\
&=&\left( \alpha -2\right) \left\vert w\right\vert ^{\alpha -4}w.
\end{eqnarray*}%
and%
\begin{eqnarray*}
\frac{\partial }{\partial w^{\ell }}\tan ^{-1}\frac{w^{\ell ^{\prime }}}{%
w^{\ell }} &=&\frac{1}{1+\left( \frac{w^{\ell ^{\prime }}}{w^{\ell }}\right)
^{2}}\frac{-w^{\ell ^{\prime }}}{\left( w^{\ell }\right) ^{2}}=\frac{%
-w^{\ell ^{\prime }}}{\left\vert w\right\vert ^{2}}, \\
\frac{\partial }{\partial w^{\ell ^{\prime }}}\tan ^{-1}\frac{w^{\ell
^{\prime }}}{w^{\ell }} &=&\frac{1}{1+\left( \frac{w^{\ell ^{\prime }}}{%
w^{\ell }}\right) ^{2}}\frac{1}{w^{\ell }}=\frac{w^{\ell }}{\left\vert
w\right\vert ^{2}}.
\end{eqnarray*}%
to calculate that the gradient of the convolution kernel%
\begin{equation*}
K_{\ell }^{\alpha }\left( w\right) =\frac{\Omega _{\ell }\left( w\right) }{%
\left\vert w\right\vert ^{2-\alpha }}=\frac{\Omega \left( \theta _{\ell
}\left( w\right) \right) }{\left\vert w\right\vert ^{2-\alpha }}=\frac{%
\Omega \left( \tan ^{-1}\frac{w^{\ell ^{\prime }}}{w^{\ell }}\right) }{%
\left\vert w\right\vert ^{2-\alpha }},
\end{equation*}%
is given by, 
\begin{eqnarray*}
\nabla K_{\ell }^{\alpha }\left( w\right)  &=&\nabla \left( \frac{\Omega
_{\ell }\left( w\right) }{\left\vert w\right\vert ^{2-\alpha }}\right)
=\Omega \left( \theta _{\ell }\left( w\right) \right) \nabla \left\vert
w\right\vert ^{\alpha -2}+\left\vert w\right\vert ^{\alpha -2}\Omega
^{\prime }\left( \theta _{\ell }\left( w\right) \right) \nabla \theta _{\ell
} \\
&=&\frac{\left( \alpha -2\right) \Omega \left( \theta _{\ell }\left(
w\right) \right) \ w+\Omega ^{\prime }\left( \theta _{\ell }\left( w\right)
\right) \ w^{\perp }}{\left\vert w\right\vert ^{4-\alpha }}.
\end{eqnarray*}%
Thus the linear part $\Lambda _{\ell }^{\alpha }$ in the Taylor expansion of 
$T_{\ell }^{\alpha }\mu $ is given by%
\begin{equation*}
\Lambda _{\ell }^{\alpha }=\left( x-z\right) \cdot \int \nabla K_{\ell
}^{\alpha }\left( c_{J}-y\right) d\mu \left( y\right) \equiv \left(
x-z\right) \cdot \mathbf{Z}_{\Omega _{\ell }}^{\alpha }\left( c_{J};\mu
\right) ,
\end{equation*}%
where 
\begin{eqnarray*}
\mathbf{Z}_{\Omega _{\ell }}^{\alpha }\left( c_{J};\mu \right)  &=&\int_{%
\mathbb{R}^{2}}\frac{\left( \alpha -2\right) \Omega \left( \theta _{\ell
}\left( c_{J}-y\right) \right) \ \left( c_{J}\mathbf{-}y\right) +\Omega
^{\prime }\left( \theta _{\ell }\left( c_{J}-y\right) \right) \ \left( c_{J}%
\mathbf{-}y\right) ^{\perp }}{\left\vert c_{J}-y\right\vert ^{4-\alpha }}%
d\mu \left( y\right)  \\
&=&\int_{w\in \mathbb{S}^{1}}\left\{ \left( \alpha -2\right) \Omega \left(
\theta _{\ell }\left( w\right) \right) w^{1}-\Omega ^{\prime }\left( \theta
_{\ell }\left( w\right) \right) w^{2}\right\} \mathbf{e}^{1}d\Psi _{\mu
}\left( w\right)  \\
&&+\int_{w\in \mathbb{S}^{1}}\left\{ \left( \alpha -2\right) \Omega \left(
\theta _{\ell }\left( w\right) \right) w^{2}+\Omega ^{\prime }\left( \theta
_{\ell }\left( w\right) \right) w^{1}\right\} \mathbf{e}^{2}d\Psi _{\mu
}\left( w\right) ,
\end{eqnarray*}%
and $\mathbf{e}^{\ell }$ is the coordinate vector with a\ $1$ in the $\ell
^{th}$ position. Here the measure $\Psi _{\mu }$ is an essentially \emph{%
arbitrary} positive finite measure on the circle $\mathbb{S}^{1}$ given
formally by%
\begin{equation*}
d\Psi _{\mu }\left( w\right) =\int_{0}^{\infty }r^{\alpha -3}d\mu _{w}\left(
r\right) =\int_{0}^{\infty }r^{\alpha -3}d\mu \left( rw\right) ,\ \ \ \ \
w\in \mathbb{S}^{1}.
\end{equation*}

We use,%
\begin{eqnarray*}
\tan \theta _{\ell }\left( w\right)  &=&\frac{\left( -1\right) ^{\ell
^{\prime }}w^{\ell ^{\prime }}}{w^{\ell }}, \\
\csc \theta _{\ell }\left( w\right)  &=&\left( -1\right) ^{\ell ^{\prime }}%
\sqrt{1+\cot ^{2}\theta _{\ell }\left( w\right) }=\left( -1\right) ^{\ell
^{\prime }}\sqrt{1+\left( \frac{w^{\ell }}{w^{\ell ^{\prime }}}\right) ^{2}}=%
\frac{\left\vert w\right\vert }{\left( -1\right) ^{\ell ^{\prime }}w^{\ell
^{\prime }}}, \\
\sin \theta _{\ell }\left( w\right)  &=&\frac{\left( -1\right) ^{\ell
^{\prime }}w^{\ell ^{\prime }}}{\left\vert w\right\vert }\text{ and }\cos
\theta _{\ell }\left( w\right) =\frac{w^{\ell }}{\left\vert w\right\vert },
\end{eqnarray*}%
for $w\neq 0$, to obtain%
\begin{eqnarray*}
\mathbf{Z}_{\Omega _{1}}^{\alpha }\left( c_{J};\mu \right)  &=&\int_{\mathbb{%
S}^{1}}\left\{ \left( \alpha -2\right) \Omega \left( \theta _{1}\left(
w\right) \right) \cos \theta _{1}\left( w\right) -\Omega ^{\prime }\left(
\theta _{1}\left( w\right) \right) \sin \theta _{1}\left( w\right) \right\} 
\mathbf{e}^{1}d\Psi _{\mu }\  \\
&&+\int_{\mathbb{S}^{1}}\int \left\{ \left( \alpha -2\right) \Omega \left(
\theta _{1}\left( w\right) \right) \sin \theta _{1}\left( w\right) +\Omega
^{\prime }\left( \theta _{1}\left( w\right) \right) \cos \theta _{1}\left(
w\right) \right\} \mathbf{e}^{2}d\Psi _{\mu } \\
&\equiv &\int_{\mathbb{S}^{1}}\left\{ A_{\alpha }^{1}\left( \theta
_{1}\left( w\right) \right) \mathbf{e}^{1}+B_{\alpha }^{1}\left( \theta
_{1}\left( w\right) \right) \mathbf{e}^{2}\right\} d\Psi _{\mu }\ ,
\end{eqnarray*}%
and%
\begin{eqnarray*}
\mathbf{Z}_{\Omega _{2}}^{\alpha }\left( c_{J};\mu \right)  &=&\int_{\mathbb{%
S}^{1}}\left\{ -\left( \alpha -2\right) \Omega \left( \theta _{2}\left(
w\right) \right) \sin \theta _{2}\left( w\right) -\Omega ^{\prime }\left(
\theta _{2}\left( w\right) \right) \cos \theta _{2}\left( w\right) \right\} 
\mathbf{e}^{1}d\Psi _{\mu }\  \\
&&+\int_{\mathbb{S}^{1}}\int \left\{ \left( \alpha -2\right) \Omega \left(
\theta _{2}\left( w\right) \right) \cos \theta _{2}\left( w\right) -\Omega
^{\prime }\left( \theta _{2}\left( w\right) \right) \sin \theta _{2}\left(
w\right) \right\} \mathbf{e}^{2}d\Psi _{\mu } \\
&\equiv &\int_{\mathbb{S}^{1}}\left\{ A_{\alpha }^{2}\left( \theta
_{2}\left( w\right) \right) \mathbf{e}^{1}+B_{\alpha }^{2}\left( \theta
_{2}\left( w\right) \right) \mathbf{e}^{2}\right\} d\Psi _{\mu }\ ,
\end{eqnarray*}%
with%
\begin{eqnarray}
A_{\alpha }^{1}\left( t\right)  &=&\left( \alpha -2\right) \Omega \left(
t\right) \cos t-\Omega ^{\prime }\left( t\right) \sin t=B_{\alpha
}^{2}\left( t\right) ,  \label{def A and B} \\
B_{\alpha }^{1}\left( t\right)  &=&\left( \alpha -2\right) \Omega \left(
t\right) \sin t+\Omega ^{\prime }\left( t\right) \cos t=-A_{\alpha
}^{2}\left( t\right) .  \notag
\end{eqnarray}

Now we show below in (\ref{strong fails}) that a necessary condition for
reversal of energy on $J$ is that the span of the pair of vectors $\left\{ 
\mathbf{Z}_{\Omega _{\ell }}^{\alpha }\left( c_{J};\mu \right) \right\}
_{\ell =1}^{2}$ is all of $\mathbb{R}^{2}$: 
\begin{equation}
\limfunc{Span}\left\{ \mathbf{Z}_{\Omega _{\ell }}^{\alpha }\left( c_{J};\mu
\right) \right\} _{\ell =1}^{2}=\mathbb{R}^{2}.  \label{choice of points}
\end{equation}%
So it suffices to show the failure of (\ref{choice of points}), i.e. that $%
\mathbf{Z}_{\Omega _{1}}^{\alpha }\left( c_{J};\mu \right) $ and $\mathbf{Z}%
_{\Omega _{2}}^{\alpha }\left( c_{J};\mu \right) $ are parallel.

At this point we take $\ell =1$ and set $\theta =\theta _{1}\left( w\right) $
so that we obtain%
\begin{eqnarray}
A_{\alpha }\left( \theta \right) &\equiv &A_{\alpha }^{1}\left( \theta
_{1}\left( w\right) \right) =\left( \alpha -2\right) \Omega \left( \theta
\right) \cos \theta -\Omega ^{\prime }\left( \theta \right) \sin \theta ,
\label{def A and B with theta 1} \\
B_{\alpha }\left( \theta \right) &\equiv &B_{\alpha }^{1}\left( \theta
_{1}\left( w\right) \right) =\left( \alpha -2\right) \Omega \left( \theta
\right) \sin \theta +\Omega ^{\prime }\left( \theta \right) \cos \theta . 
\notag
\end{eqnarray}%
In the case $\alpha =1$ these coefficients are perfect derivatives,%
\begin{eqnarray*}
A_{1}\left( \theta \right) &=&-\Omega \left( \theta \right) \cos \theta
-\Omega ^{\prime }\left( \theta \right) \sin \theta =-\left[ \Omega \left(
\theta \right) \sin \theta \right] ^{\prime }, \\
B_{1}\left( \theta \right) &=&-\Omega \left( \theta \right) \sin \theta
+\Omega ^{\prime }\left( \theta \right) \cos \theta =-\left[ \Omega \left(
\theta \right) \cos \theta \right] ^{\prime },
\end{eqnarray*}%
and so have vanishing integral on the circle. Thus with the choice $d\Psi
_{\mu }\left( \theta \right) =d\theta $ we have%
\begin{equation*}
\mathbf{Z}_{\Omega }\left( c_{J};\mu \right) =\int_{\mathbb{S}^{1}}\left\{
A_{1}\left( \theta \right) \mathbf{e}^{1}+B_{1}\left( \theta \right) \mathbf{%
e}^{2}\right\} d\theta =\mathbf{0}
\end{equation*}%
the zero vector, for \textbf{every} choice of differentiable $\Omega $ on
the circle\textbf{.}

In the case $0\leq \alpha <2$ with $\alpha \neq 1$, it is no longer possible
to find a nontrivial measure $\mu $ so that $\mathbf{Z}_{\Omega }^{\alpha
}\left( c_{J};\mu \right) $ vanishes for all differentiable $\Omega $, but
we will see that we \emph{can} always find a positive measure $\mu $ such
that the vectors $\mathbf{Z}_{\Omega _{1}}^{\alpha }\left( c_{J};\mu \right) 
$ and $\mathbf{Z}_{\Omega _{2}}^{\alpha }\left( c_{J};\mu \right) $ are
parallel for the choice $\Omega \left( \theta \right) =\cos \theta $ that
corresponds to the vector of Riesz transforms.

Indeed, in the special case that $\Omega \left( t\right) =\cos t$, and
recalling that $\theta _{2}\left( w\right) =\theta _{1}\left( w\right) +%
\frac{\pi }{2}=\theta +\frac{\pi }{2}$, we have%
\begin{eqnarray*}
A_{\alpha }^{1}\left( \theta _{1}\left( w\right) \right)  &=&A_{\alpha
}^{1}\left( \theta \right) =\left( \alpha -2\right) \cos ^{2}\theta +\sin
^{2}\theta ; \\
B_{\alpha }^{1}\left( \theta _{1}\left( w\right) \right)  &=&B_{\alpha
}^{1}\left( \theta \right) =\left( \alpha -3\right) \cos \theta \sin \theta ;
\\
A_{\alpha }^{2}\left( \theta _{2}\left( w\right) \right)  &=&-B_{\alpha
}^{1}\left( \theta +\frac{\pi }{2}\right) =-\left( \alpha -3\right) \cos
\left( \theta +\frac{\pi }{2}\right) \sin \left( \theta +\frac{\pi }{2}%
\right)  \\
&=&\left( \alpha -3\right) \cos \theta \sin \theta ; \\
B_{\alpha }^{2}\left( \theta _{2}\left( w\right) \right)  &=&A_{\alpha
}^{1}\left( \theta +\frac{\pi }{2}\right) =\left( \alpha -2\right) \cos
^{2}\left( \theta +\frac{\pi }{2}\right) +\sin ^{2}\left( \theta +\frac{\pi 
}{2}\right)  \\
&=&\left( \alpha -2\right) \sin ^{2}\theta +\cos ^{2}\theta .
\end{eqnarray*}%
Thus we also have%
\begin{eqnarray*}
\mathbf{Z}_{\Omega _{1}}^{\alpha }\left( c_{J};\mu \right)  &=&\int_{\mathbb{%
S}^{1}}\left\{ A_{\alpha }^{1}\left( \theta _{1}\left( w\right) \right) 
\mathbf{e}^{1}+B_{\alpha }^{1}\left( \theta _{1}\left( w\right) \right) 
\mathbf{e}^{2}\right\} d\Psi _{\mu } \\
&=&\int_{\mathbb{S}^{1}}\left\{ \left[ \left( \alpha -2\right) \cos
^{2}\theta +\sin ^{2}\theta \right] \mathbf{e}^{1}+\left[ \left( \alpha
-3\right) \cos \theta \sin \theta \right] \mathbf{e}^{2}\right\} d\Psi _{\mu
} \\
&=&\left\{ \int_{\mathbb{S}^{1}}\left[ \left( \alpha -2\right) \cos
^{2}\theta +\sin ^{2}\theta \right] d\Psi _{\mu }\right\} \mathbf{e}%
^{1}+\left\{ \int_{\mathbb{S}^{1}}\left[ \left( \alpha -3\right) \cos \theta
\sin \theta \right] d\Psi _{\mu }\right\} \mathbf{e}^{2}
\end{eqnarray*}%
and 
\begin{eqnarray*}
\mathbf{Z}_{\Omega _{2}}^{\alpha }\left( c_{J};\mu \right)  &=&\int_{\mathbb{%
S}^{1}}\left\{ A_{\alpha }^{2}\left( \theta _{2}\left( w\right) \right) 
\mathbf{e}^{1}+B_{\alpha }^{2}\left( \theta _{2}\left( w\right) \right) 
\mathbf{e}^{2}\right\} d\Psi _{\mu } \\
&=&\int_{\mathbb{S}^{1}}\left\{ \left[ \left( \alpha -3\right) \cos \theta
\sin \theta \right] \mathbf{e}^{1}+\left[ \left( \alpha -2\right) \sin
^{2}\theta +\cos ^{2}\theta \right] \mathbf{e}^{2}\right\} d\Psi _{\mu } \\
&=&\left\{ \int_{\mathbb{S}^{1}}\left[ \left( \alpha -3\right) \cos \theta
\sin \theta \right] d\Psi _{\mu }\right\} \mathbf{e}^{1}+\left\{ \int_{%
\mathbb{S}^{1}}\left[ \left( \alpha -2\right) \sin ^{2}\theta +\cos
^{2}\theta \right] d\Psi _{\mu }\right\} \mathbf{e}^{2}.
\end{eqnarray*}%
Using%
\begin{eqnarray}
&&\left( \alpha -2\right) \cos ^{2}\theta +\sin ^{2}\theta =\left( \alpha
-3\right) \cos ^{2}\theta +1,  \label{identities} \\
&&\left( \alpha -2\right) \sin ^{2}\theta +\cos ^{2}\theta =\left( \alpha
-3\right) \sin ^{2}\theta +1,  \notag \\
&&\sin \theta \cos \theta =\frac{1}{2}\sin 2\theta ,\ \cos ^{2}\theta =\frac{%
1+\cos 2\theta }{2},\ \sin ^{2}\theta =\frac{1-\cos 2\theta }{2},  \notag
\end{eqnarray}%
we see that%
\begin{eqnarray*}
\left( \alpha -2\right) \cos ^{2}\theta +\sin ^{2}\theta  &=&\left( \alpha
-3\right) \frac{1+\cos 2\theta }{2}+1=\frac{\alpha -3}{2}\cos 2\theta +\frac{%
\alpha -1}{2}, \\
\left( \alpha -2\right) \sin ^{2}\theta +\cos ^{2}\theta  &=&\left( \alpha
-3\right) \frac{1-\cos 2\theta }{2}+1=-\frac{\alpha -3}{2}\cos 2\theta +%
\frac{\alpha -1}{2}, \\
\left( \alpha -3\right) \cos \theta \sin \theta  &=&\frac{\alpha -3}{2}\sin
2\theta .
\end{eqnarray*}%
Plugging these formulas into those for $\mathbf{Z}_{\Omega _{1}}^{\alpha
}\left( c_{J};\mu \right) $ and $\mathbf{Z}_{\Omega _{2}}^{\alpha }\left(
c_{J};\mu \right) $ we obtain 
\begin{eqnarray*}
&&\det \left[ 
\begin{array}{c}
\mathbf{Z}_{\Omega _{1}}^{\alpha }\left( c_{J};\mu \right)  \\ 
\mathbf{Z}_{\Omega _{2}}^{\alpha }\left( c_{J};\mu \right) 
\end{array}%
\right]  \\
&=&\det \left[ 
\begin{array}{cc}
\int_{\mathbb{S}^{1}}\left[ \frac{\alpha -3}{2}\cos 2\theta +\frac{\alpha -1%
}{2}\right] d\Psi _{\mu } & \int_{\mathbb{S}^{1}}\left[ \frac{\alpha -3}{2}%
\sin 2\theta \right] d\Psi _{\mu } \\ 
\int_{\mathbb{S}^{1}}\left[ \frac{\alpha -3}{2}\sin 2\theta \right] d\Psi
_{\mu } & \int_{\mathbb{S}^{1}}\left[ -\frac{\alpha -3}{2}\cos 2\theta +%
\frac{\alpha -1}{2}\right] d\Psi _{\mu }%
\end{array}%
\right]  \\
&=&\left( \frac{\alpha -3}{2}\int_{\mathbb{S}^{1}}\cos 2\theta d\Psi _{\mu }+%
\frac{\alpha -1}{2}\left\Vert \Psi _{\mu }\right\Vert \right) \left( -\frac{%
\alpha -3}{2}\int_{\mathbb{S}^{1}}\cos 2\theta d\Psi _{\mu }+\frac{\alpha -1%
}{2}\left\Vert \Psi _{\mu }\right\Vert \right)  \\
&&-\left( \frac{\alpha -3}{2}\int_{\mathbb{S}^{1}}\sin 2\theta d\Psi _{\mu
}\right) ^{2} \\
&=&\left( \frac{\alpha -1}{2}\left\Vert \Psi _{\mu }\right\Vert \right)
^{2}-\left\{ \left( \frac{\alpha -3}{2}\int_{\mathbb{S}^{1}}\cos 2\theta
d\Psi _{\mu }\right) ^{2}+\left( \frac{\alpha -3}{2}\int_{\mathbb{S}%
^{1}}\sin 2\theta d\Psi _{\mu }\right) ^{2}\right\} .
\end{eqnarray*}%
Thus $\det \left[ 
\begin{array}{c}
\mathbf{Z}_{\Omega _{1}}^{\alpha }\left( c_{J};\mu \right)  \\ 
\mathbf{Z}_{\Omega _{2}}^{\alpha }\left( c_{J};\mu \right) 
\end{array}%
\right] =0$ if and only if the length of the vector%
\begin{equation*}
\frac{\alpha -3}{2}\left( 
\begin{array}{c}
\int_{\mathbb{S}^{1}}\cos 2\theta d\Psi _{\mu } \\ 
\int_{\mathbb{S}^{1}}\sin 2\theta d\Psi _{\mu }%
\end{array}%
\right) 
\end{equation*}%
equals $\frac{\left\vert \alpha -1\right\vert }{2}\left\Vert \Psi _{\mu
}\right\Vert $, i.e.%
\begin{equation}
\left\Vert \left( 
\begin{array}{c}
\int_{\mathbb{S}^{1}}\cos 2\theta d\Psi _{\mu } \\ 
\int_{\mathbb{S}^{1}}\sin 2\theta d\Psi _{\mu }%
\end{array}%
\right) \right\Vert =\frac{\left\vert \alpha -1\right\vert }{\left\vert
\alpha -3\right\vert }\left\Vert \Psi _{\mu }\right\Vert \ .
\label{the condition}
\end{equation}

To construct a positive probability measure $d\Psi _{\mu }$ on the circle
that satisfies (\ref{the condition}), we first observe that if $d\Psi _{\mu
}=\delta _{0}$ is the unit point mass at $0$, then%
\begin{equation*}
\left\Vert \left( 
\begin{array}{c}
\int_{\mathbb{S}^{1}}\cos 2\theta d\Psi _{\mu } \\ 
\int_{\mathbb{S}^{1}}\sin 2\theta d\Psi _{\mu }%
\end{array}%
\right) \right\Vert =\left\Vert \left( 
\begin{array}{c}
\int_{\mathbb{S}^{1}}d\Psi _{\mu } \\ 
0%
\end{array}%
\right) \right\Vert =\left\Vert \Psi _{\mu }\right\Vert ,
\end{equation*}%
and since $\left\vert \alpha -1\right\vert <\left\vert \alpha -3\right\vert $
for all $0\leq \alpha <2$, we have%
\begin{equation*}
\left\Vert \left( 
\begin{array}{c}
\int_{\mathbb{S}^{1}}\cos 2\theta d\Psi _{\mu } \\ 
\int_{\mathbb{S}^{1}}\sin 2\theta d\Psi _{\mu }%
\end{array}%
\right) \right\Vert >\frac{\left\vert \alpha -1\right\vert }{\left\vert
\alpha -3\right\vert }\left\Vert \Psi _{\mu }\right\Vert ,
\end{equation*}%
in this case. On the other hand, if $d\Psi _{\mu }\left( \theta \right) =%
\frac{1}{2\pi }d\theta $ is normalized Lebesgue measure on the circle, we
have%
\begin{equation*}
\left\Vert \left( 
\begin{array}{c}
\int_{\mathbb{S}^{1}}\cos 2\theta d\Psi _{\mu } \\ 
\int_{\mathbb{S}^{1}}\sin 2\theta d\Psi _{\mu }%
\end{array}%
\right) \right\Vert =\left\Vert \left( 
\begin{array}{c}
0 \\ 
0%
\end{array}%
\right) \right\Vert =0<\frac{\left\vert \alpha -1\right\vert }{\left\vert
\alpha -3\right\vert }\left\Vert \Psi _{\mu }\right\Vert .
\end{equation*}
It is now easy to see that there is a convex combination $d\Psi _{\mu
}=\left( 1-\lambda \right) \delta _{0}+\lambda \frac{1}{2\pi }d\theta $ such
that (\ref{the condition}) holds. Thus (\ref{choice of points}) fails, and
we now show that energy reversal fails.

In fact, we may assume that both $\mathbf{Z}_{\Omega _{1}}^{\alpha }\left(
c_{J};\mu \right) $ and $\mathbf{Z}_{\Omega _{2}}^{\alpha }\left( c_{J};\mu
\right) $ are parallel to the coordinate vector $\mathbf{e}_{2}$, and in
this case we will see that we can reverse at most the coordinate energy $%
\mathsf{E}^{2}\left( J,\omega \right) $, defined above by%
\begin{equation*}
\mathsf{E}^{2}\left( J,\omega \right) ^{2}\equiv \frac{1}{\left\vert
J\right\vert _{\omega }}\frac{1}{\left\vert J\right\vert _{\omega }}%
\int_{J}\int_{J}\left\vert \frac{x^{2}-z^{2}}{\left\vert J\right\vert ^{%
\frac{1}{n}}}\right\vert ^{2}d\omega \left( x\right) d\omega \left( z\right)
,
\end{equation*}%
and not the full energy $\mathsf{E}\left( J,\omega \right) $. More
precisely, we claim that there is a measure $\omega $ such that for $\gamma $
so large that $\varepsilon \ll C_{0}$, the strong reversal of $\omega $%
-energy inequality (\ref{will fail}) fails. Indeed, using that $\mathbf{Z}%
_{\Omega _{\ell }}^{\alpha }\left( c_{J}\right) \left( c_{J}\right) $ is
parallel to $\mathbf{e}^{2}$, we have that%
\begin{eqnarray}
&&\int_{J}\int_{J}\left\vert \mathbf{T}^{\alpha }\mu \left( x\right) -%
\mathbf{T}^{\alpha }\mu \left( z\right) \right\vert ^{2}d\omega \left(
x\right) d\omega \left( z\right)   \label{strong fails} \\
&=&\sum_{\ell =1}^{2}\int_{J}\int_{J}\left\vert \left( x-z\right) \cdot 
\mathbf{Z}_{\Omega _{\ell }}^{\alpha }\left( c_{J}\right) +\left[ E_{\ell
,x}^{\alpha }-E_{\ell ,z}^{\alpha }\right] \right\vert ^{2}d\omega \left(
x\right) d\omega \left( z\right)   \notag \\
&\leq &\sum_{\ell =1}^{2}\int_{J}\int_{J}\left\vert \frac{\mathrm{P}^{\alpha
}\left( J,\mu \right) }{\left\vert J\right\vert ^{\frac{1}{2}}}\left(
x-z\right) \cdot \frac{\mathbf{Z}_{\Omega _{\ell }}^{\alpha }\left(
c_{J}\right) \left( c_{J}\right) }{\left\vert \mathbf{Z}_{\Omega _{\ell
}}^{\alpha }\left( c_{J}\right) \left( c_{J}\right) \right\vert }\right\vert
^{2}d\omega \left( x\right) d\omega \left( z\right)   \notag \\
&&+C\sum_{\ell =1}^{2}\int_{J}\int_{J}\left\vert \varepsilon \frac{\mathrm{P}%
^{\alpha }\left( J,\mu \right) }{\left\vert J\right\vert ^{\frac{1}{2}}}%
\left\vert x-z\right\vert \right\vert ^{2}d\omega \left( x\right) d\omega
\left( z\right)   \notag \\
&\leq &\mathsf{E}^{2}\left( J,\omega \right) ^{2}\mathrm{P}^{\alpha }\left(
J,\mu \right) ^{2}+C\varepsilon ^{2}\mathsf{E}\left( J,\omega \right) ^{2}%
\mathrm{P}^{\alpha }\left( J,\mu \right) ^{2}  \notag \\
&\leq &\frac{1}{10}C_{0}\mathsf{E}\left( J,\omega \right) ^{2}\mathrm{P}%
^{\alpha }\left( J,\mu \right) ^{2},  \notag
\end{eqnarray}%
provided we choose $\gamma $ so large that $C\varepsilon ^{2}\leq \frac{1}{10%
}C_{0}$ and provided we choose $\omega $ so that $\mathsf{E}^{2}\left(
J,\omega \right) =0$ but $\mathsf{E}\left( J,\omega \right) >0$. This
completes the proof of the first assertion in Lemma \ref{LRE}.
\end{proof}

\begin{remark}
The condition (\ref{the condition}) must be invariant under rotations, i.e.
invariant under replacing $\theta $ by $\theta -\phi $ for any constant $%
\phi $, and this is easily seen using (\ref{identities}) above:%
\begin{eqnarray*}
\left( 
\begin{array}{c}
\int_{\mathbb{S}^{1}}\cos 2\left( \theta -\phi \right) d\Psi _{\mu } \\ 
\int_{\mathbb{S}^{1}}\sin 2\left( \theta -\phi \right) d\Psi _{\mu }%
\end{array}%
\right) &=&\left( 
\begin{array}{c}
\cos 2\phi \int_{\mathbb{S}^{1}}\cos 2\theta d\Psi _{\mu }+\sin 2\phi \int_{%
\mathbb{S}^{1}}\sin 2\theta d\Psi _{\mu } \\ 
\cos 2\phi \int_{\mathbb{S}^{1}}\sin 2\theta d\Psi _{\mu }-\sin 2\phi \int_{%
\mathbb{S}^{1}}\cos 2\theta d\Psi _{\mu }%
\end{array}%
\right) \\
&=&\cos 2\phi \left( 
\begin{array}{c}
\int_{\mathbb{S}^{1}}\cos 2\theta d\Psi _{\mu } \\ 
\int_{\mathbb{S}^{1}}\sin 2\theta d\Psi _{\mu }%
\end{array}%
\right) -\sin 2\phi \left( 
\begin{array}{c}
\int_{\mathbb{S}^{1}}\cos 2\theta d\Psi _{\mu } \\ 
\int_{\mathbb{S}^{1}}\sin 2\theta d\Psi _{\mu }%
\end{array}%
\right) ^{\perp },
\end{eqnarray*}%
which has length independent of $\phi $.
\end{remark}

\begin{remark}
The above proof shows that for each $t\in \mathbb{R}$, the convolution
kernel 
\begin{equation*}
\Phi _{\alpha ,t}\left( x,y\right) =\frac{x\cos t+y\sin t}{\left(
x^{2}+y^{2}\right) ^{\frac{3-\alpha }{2}}},
\end{equation*}%
in the plane with coordinates $\left( x,y\right) $, $x,y\in \mathbb{R}$, and
the probability measure $d\mu _{\alpha }$ supported on the circle $\mathbb{S}%
^{1}=\left[ 0,2\pi \right) $ given by%
\begin{equation*}
d\mu _{\alpha }\left( \theta \right) =\frac{\left\vert \alpha -1\right\vert 
}{\left\vert \alpha -3\right\vert }\delta _{0}\left( \theta \right) +\frac{%
\left\vert \alpha -3\right\vert -\left\vert \alpha -1\right\vert }{%
\left\vert \alpha -3\right\vert }\frac{d\theta }{2\pi },
\end{equation*}%
satisfy the property that $\func{grad}\left( \Phi _{\alpha ,t}\ast \mu
_{\alpha }\right) \left( 0,0\right) $ points in the same direction for all $%
t $. A direct calculation shows that 
\begin{equation*}
\func{grad}\left( \Phi _{\alpha ,t}\ast \mu _{\alpha }\right) \left(
0,0\right) =\left( \alpha -1\right) \left\{ 
\begin{array}{ccc}
\left[ \cos t,0\right] & \text{ for } & 0\leq \alpha <1 \\ 
\left[ 0,\sin t\right] & \text{ for } & 1<\alpha <2%
\end{array}%
\right. .
\end{equation*}%
Indeed, if for $\theta \in \mathbb{R}$ we define $\Phi _{\alpha ,t}^{\theta
} $ to be the convolution of $\Phi _{\alpha ,t}$ with the unit point mass $%
\delta _{e^{i\theta }}$ at $e^{i\theta }$ in the circle, 
\begin{equation*}
\Phi _{\alpha ,t}^{\theta }\left( x,y\right) \equiv \left( \Phi _{\alpha
,t}\ast \delta _{e^{i\theta }}\right) \left( x,y\right) =\frac{\left( x-\cos
\theta \right) \cos t+\left( y-\sin \theta \right) \sin t}{\left( \left(
x-\cos \theta \right) ^{2}+\left( y-\sin \theta \right) ^{2}\right) ^{\frac{%
3-\alpha }{2}}},
\end{equation*}%
then we have%
\begin{eqnarray*}
\func{grad}\Phi _{\alpha ,t}^{\theta }\left( x,y\right) &=&\left[ \left( 
\frac{\partial }{\partial x}\Phi _{\alpha ,t}^{\theta }\right) \left(
x,y\right) ,\left( \frac{\partial }{\partial y}\Phi _{\alpha ,t}^{\theta
}\right) \left( x,y\right) \right] \\
&=&\frac{\left[ \cos t,\sin t\right] }{\left( \left( x-\cos \theta \right)
^{2}+\left( y-\sin \theta \right) ^{2}\right) ^{\frac{3-\alpha }{2}}} \\
&&-\frac{3-\alpha }{2}\left\{ \left( x-\cos \theta \right) \cos t+\left(
y-\sin \theta \right) \sin t\right\} \frac{\left[ 2\left( x-\cos \theta
\right) ,2\left( y-\sin \theta \right) \right] }{\left( \left( x-\cos \theta
\right) ^{2}+\left( y-\sin \theta \right) ^{2}\right) ^{\frac{5-\alpha }{2}}}%
,
\end{eqnarray*}%
and when $\left( x,y\right) =\left( 0,0\right) $ we get%
\begin{equation*}
\func{grad}\Phi _{\alpha ,t}^{\theta }\left( 0,0\right) =\left[ \cos t,\sin t%
\right] -\left( 3-\alpha \right) \left\{ \cos \theta \cos t+\sin \theta \sin
t\right\} \left[ \cos \theta ,\sin \theta \right] .
\end{equation*}%
Thus we have%
\begin{eqnarray*}
\func{grad}\Phi _{t}^{0}\left( 0,0\right) &=&\func{grad}\Phi _{t}^{0}\left(
0,0\right) =\left[ \cos t,\sin t\right] -\left( 3-\alpha \right) \cos t\left[
1,0\right] \\
&=&\left[ -\left( 2-\alpha \right) \cos t,\sin t\right] ,
\end{eqnarray*}%
and%
\begin{eqnarray*}
\func{grad}\left( \Phi _{t}\ast \frac{d\theta }{2\pi }\right) \left(
0,0\right) &=&\func{grad}\int_{0}^{2\pi }\left( \Phi _{t}\ast \delta
_{e^{i\theta }}\right) \left( 0,0\right) \frac{d\theta }{2\pi } \\
&=&\left[ \cos t,\sin t\right] -\frac{3-\alpha }{2}\left[ \cos t,\sin t%
\right] \\
&=&\left[ \frac{\alpha -1}{2}\cos t,\frac{\alpha -1}{2}\sin t\right] .
\end{eqnarray*}%
Thus 
\begin{eqnarray*}
&&\left( 3-\alpha \right) \func{grad}\left( \Phi _{\alpha ,t}\ast \mu
_{\alpha }\right) \left( 0,0\right) \\
&=&\left\vert \alpha -1\right\vert \left[ -\left( 2-\alpha \right) \cos
t,\sin t\right] +\left( \left\vert \alpha -3\right\vert -\left\vert \alpha
-1\right\vert \right) \left[ \frac{\alpha -1}{2}\cos t,\frac{\alpha -1}{2}%
\sin t\right] \\
&=&\left[ \left\{ -\left( 2-\alpha \right) \left\vert \alpha -1\right\vert
+\left( \left\vert \alpha -3\right\vert -\left\vert \alpha -1\right\vert
\right) \frac{\alpha -1}{2}\right\} \cos t,\left\{ \left\vert \alpha
-1\right\vert +\left( \left\vert \alpha -3\right\vert -\left\vert \alpha
-1\right\vert \right) \frac{\alpha -1}{2}\right\} \sin t\right] .
\end{eqnarray*}%
Now for $0\leq \alpha <1$ we get%
\begin{equation*}
\left\vert \alpha -1\right\vert +\left( \left\vert \alpha -3\right\vert
-\left\vert \alpha -1\right\vert \right) \frac{\alpha -1}{2}=1-\alpha +2%
\frac{\alpha -1}{2}=0,
\end{equation*}%
and%
\begin{equation*}
-\left( 2-\alpha \right) \left\vert \alpha -1\right\vert +\left( \left\vert
\alpha -3\right\vert -\left\vert \alpha -1\right\vert \right) \frac{\alpha -1%
}{2}=\left( \alpha -1\right) \left( 3-\alpha \right) .
\end{equation*}%
For $1<\alpha <2$ we get%
\begin{equation*}
\left\vert \alpha -1\right\vert +\left( \left\vert \alpha -3\right\vert
-\left\vert \alpha -1\right\vert \right) \frac{\alpha -1}{2}=\left( \alpha
-1\right) \left( 3-\alpha \right) ,
\end{equation*}%
and%
\begin{equation*}
-\left( 2-\alpha \right) \left\vert \alpha -1\right\vert +\left( \left\vert
\alpha -3\right\vert -\left\vert \alpha -1\right\vert \right) \frac{\alpha -1%
}{2}=0.
\end{equation*}
\end{remark}

\bigskip

\begin{proof}[Proof of Lemma \protect\ref{LRE} for the vector of trig
polynomials]
Recall that with $\theta =\theta _{1}\left( w\right) $ we obtain%
\begin{eqnarray*}
A_{\alpha }\left( \theta \right) &=&\left( \alpha -2\right) \Omega \left(
\theta \right) \cos \theta -\Omega ^{\prime }\left( \theta \right) \sin
\theta \\
B_{\alpha }\left( \theta \right) &=&\left( \alpha -2\right) \Omega \left(
\theta \right) \sin \theta +\Omega ^{\prime }\left( \theta \right) \cos
\theta .
\end{eqnarray*}%
Thus we have%
\begin{eqnarray*}
A_{\alpha }\left( \theta \right) &=&\left\{ \left( \alpha -2\right) \Omega
\left( \theta \right) +i\Omega ^{\prime }\left( \theta \right) \right\}
\left\{ \cos \theta +i\sin \theta \right\} \\
&&-i\left\{ \left( \alpha -2\right) \Omega \left( \theta \right) \sin \theta
+\Omega ^{\prime }\left( \theta \right) \cos \theta \right\} \\
&=&\left\{ \left( \alpha -2\right) \Omega \left( \theta \right) +i\Omega
^{\prime }\left( \theta \right) \right\} \left\{ \cos \theta +i\sin \theta
\right\} -iB_{\alpha }\left( \theta \right) ,
\end{eqnarray*}%
and so%
\begin{equation*}
\left\{ \left( \alpha -2\right) \Omega \left( \theta \right) +i\Omega
^{\prime }\left( \theta \right) \right\} \left\{ \cos \theta +i\sin \theta
\right\} =A_{\alpha }\left( \theta \right) +iB_{\alpha }\left( \theta
\right) .
\end{equation*}%
This shows that in complex notation,%
\begin{eqnarray*}
\mathbf{Z}_{\Omega }^{\alpha }\left( c_{J};\mu \right) &=&\int_{\mathbb{S}%
^{1}}\left\{ A_{\alpha }\left( \theta \right) +iB_{\alpha }\left( \theta
\right) \right\} d\Psi _{\mu } \\
&=&\int_{\mathbb{S}^{1}}\left\{ \left( \alpha -2\right) \Omega \left( \theta
\right) +i\Omega ^{\prime }\left( \theta \right) \right\} \left\{ \cos
\theta +i\sin \theta \right\} d\Psi _{\mu } \\
&=&\int_{\mathbb{S}^{1}}\mathbf{\Omega }_{\alpha }\left( \theta \right)
e^{i\theta }d\Psi _{\mu },
\end{eqnarray*}%
where%
\begin{equation*}
\mathbf{\Omega }_{\alpha }\left( \theta \right) \equiv \left( \alpha
-2\right) \Omega \left( \theta \right) +i\Omega ^{\prime }\left( \theta
\right) .
\end{equation*}

Recall the product formulas%
\begin{eqnarray*}
2\cos A\cos B &=&\cos \left( A-B\right) +\cos \left( A+B\right) ; \\
2\sin A\sin B &=&\cos \left( A-B\right) -\cos \left( A+B\right) ; \\
2\sin A\cos B &=&\sin \left( A-B\right) +\sin \left( A+B\right) .
\end{eqnarray*}%
In the special case that $\Omega _{1}^{k}\left( t\right) =\cos kt$ we thus
have%
\begin{eqnarray*}
A_{\alpha }\left( \theta \right) &=&\left( \alpha -2\right) \cos k\theta
\cos \theta +k\sin k\theta \sin \theta \\
&=&\left( \alpha -2\right) \frac{1}{2}\left[ \cos \left( k-1\right) \theta
+\cos \left( k+1\right) \theta \right] \\
&&+k\frac{1}{2}\left[ \cos \left( k-1\right) \theta -\cos \left( k+1\right)
\theta \right] \\
&=&\left\{ \frac{\alpha +k}{2}-1\right\} \cos \left( k-1\right) \theta
+\left\{ \frac{\alpha -k}{2}-1\right\} \cos \left( k+1\right) \theta ; \\
B_{\alpha }\left( \theta \right) &=&\left( \alpha -2\right) \cos k\theta
\sin \theta -k\sin k\theta \cos \theta \\
&=&\left( \alpha -2\right) \frac{1}{2}\left[ -\sin \left( k-1\right) \theta
+\sin \left( k+1\right) \theta \right] \\
&&-k\frac{1}{2}\left[ \sin \left( k-1\right) \theta +\sin \left( k+1\right)
\theta \right] \\
&=&-\left\{ \frac{\alpha +k}{2}-1\right\} \sin \left( k-1\right) \theta
+\left\{ \frac{\alpha -k}{2}-1\right\} \sin \left( k+1\right) \theta ,
\end{eqnarray*}%
and so%
\begin{eqnarray*}
\mathbf{Z}_{\Omega _{1}^{k}}^{\alpha }\left( c_{J};\mu \right) &=&\int_{%
\mathbb{S}^{1}}\left\{ A_{\alpha }\left( \theta \right) \mathbf{e}%
^{1}+B_{\alpha }\left( \theta \right) \mathbf{e}^{2}\right\} d\Psi _{\mu } \\
&=&\int_{\mathbb{S}^{1}}\left[ \left\{ \frac{\alpha +k}{2}-1\right\} \cos
\left( k-1\right) \theta +\left\{ \frac{\alpha -k}{2}-1\right\} \cos \left(
k+1\right) \theta \right] d\Psi _{\mu }\ \mathbf{e}^{1} \\
&&+\int_{\mathbb{S}^{1}}\left[ -\left\{ \frac{\alpha +k}{2}-1\right\} \sin
\left( k-1\right) \theta +\left\{ \frac{\alpha -k}{2}-1\right\} \sin \left(
k+1\right) \theta \right] d\Psi _{\mu }\ \mathbf{e}^{2} \\
&=&\left\{ \frac{\alpha +k-2}{2}\right\} \int_{\mathbb{S}^{1}}\left( 
\begin{array}{c}
\cos \left( k-1\right) \theta \\ 
-\sin \left( k-1\right) \theta%
\end{array}%
\right) d\Psi _{\mu } \\
&&+\left\{ \frac{\alpha -k-2}{2}\right\} \int_{\mathbb{S}^{1}}\left( 
\begin{array}{c}
\cos \left( k+1\right) \theta \\ 
\sin \left( k+1\right) \theta%
\end{array}%
\right) d\Psi _{\mu } \\
&=&\int_{\mathbb{S}^{1}}\left\{ \left( \frac{\alpha +k-2}{2}\right)
e^{-i\left( k-1\right) \theta }+\left( \frac{\alpha -k-2}{2}\right)
e^{i\left( k+1\right) \theta }\right\} d\Psi _{\mu } \\
&=&\left( \frac{\alpha +k-2}{2}\right) \overline{\widehat{\Psi _{\mu }}%
\left( k-1\right) }+\left( \frac{\alpha -k-2}{2}\right) \widehat{\Psi _{\mu }%
}\left( k+1\right) \ .
\end{eqnarray*}

Next we take $\Omega _{2}^{k}\left( \theta \right) =\sin k\theta $ so that%
\begin{eqnarray*}
A_{\alpha }\left( \theta \right) &=&\left( \alpha -2\right) \sin k\theta
\cos \theta -k\cos k\theta \sin \theta \\
&=&\left( \alpha -2\right) \frac{1}{2}\left[ \sin \left( k-1\right) \theta
+\sin \left( k+1\right) \theta \right] \\
&&-k\frac{1}{2}\left[ -\sin \left( k-1\right) \theta +\sin \left( k+1\right)
\theta \right] \\
&=&\left\{ \frac{\alpha +k}{2}-1\right\} \sin \left( k-1\right) \theta
+\left\{ \frac{\alpha -k}{2}-1\right\} \sin \left( k+1\right) \theta ; \\
B_{\alpha }\left( \theta \right) &=&\left( \alpha -2\right) \sin k\theta
\sin \theta +k\cos k\theta \cos \theta \\
&=&\left( \alpha -2\right) \frac{1}{2}\left[ \cos \left( k-1\right) \theta
-\cos \left( k+1\right) \theta \right] \\
&&+k\frac{1}{2}\left[ \cos \left( k-1\right) \theta +\cos \left( k+1\right)
\theta \right] \\
&=&\left\{ \frac{\alpha +k}{2}-1\right\} \cos \left( k-1\right) \theta
-\left\{ \frac{\alpha -k}{2}-1\right\} \cos \left( k+1\right) \theta .
\end{eqnarray*}%
Thus with $\Omega _{2}^{k}\left( \theta \right) =\sin k\theta $ we obtain%
\begin{eqnarray*}
\mathbf{Z}_{\Omega _{2}^{k}}^{\alpha }\left( c_{J};\mu \right) &=&\int_{%
\mathbb{S}^{1}}\left\{ A_{\alpha }\left( \theta \right) \mathbf{e}%
^{1}+B_{\alpha }\left( \theta \right) \mathbf{e}^{2}\right\} d\Psi _{\mu } \\
&=&\int_{\mathbb{S}^{1}}\left[ \left\{ \frac{\alpha +k}{2}-1\right\} \sin
\left( k-1\right) \theta +\left\{ \frac{\alpha -k}{2}-1\right\} \sin \left(
k+1\right) \theta \right] d\Psi _{\mu }\ \mathbf{e}^{1} \\
&&+\int_{\mathbb{S}^{1}}\left[ \left\{ \frac{\alpha +k}{2}-1\right\} \cos
\left( k-1\right) \theta -\left\{ \frac{\alpha -k}{2}-1\right\} \cos \left(
k+1\right) \theta \right] d\Psi _{\mu }\ \mathbf{e}^{2} \\
&=&\left\{ \frac{\alpha +k-2}{2}\right\} \int_{\mathbb{S}^{1}}\left( 
\begin{array}{c}
\sin \left( k-1\right) \theta \\ 
\cos \left( k-1\right) \theta%
\end{array}%
\right) d\Psi _{\mu } \\
&&+\left\{ \frac{\alpha -k-2}{2}\right\} \int_{\mathbb{S}^{1}}\left( 
\begin{array}{c}
\sin \left( k+1\right) \theta \\ 
-\cos \left( k+1\right) \theta%
\end{array}%
\right) d\Psi _{\mu } \\
&=&\int_{\mathbb{S}^{1}}\left\{ \left( \frac{\alpha +k-2}{2}\right)
ie^{-i\left( k-1\right) \theta }-\left( \frac{\alpha -k-2}{2}\right)
ie^{i\left( k+1\right) \theta }\right\} d\Psi _{\mu } \\
&=&i\left( \frac{\alpha +k-2}{2}\right) \overline{\widehat{\Psi _{\mu }}%
\left( k-1\right) }-i\left( \frac{\alpha -k-2}{2}\right) \widehat{\Psi _{\mu
}}\left( k+1\right) \ .
\end{eqnarray*}%
Altogether we have%
\begin{eqnarray}
\mathbf{Z}_{\Omega _{1}^{k}}^{\alpha }\left( c_{J};\mu \right) &=&\left( 
\frac{\alpha +k-2}{2}\right) \overline{\widehat{\Psi _{\mu }}\left(
k-1\right) }+\left( \frac{\alpha -k-2}{2}\right) \widehat{\Psi _{\mu }}%
\left( k+1\right) ;  \label{altogether} \\
\mathbf{Z}_{\Omega _{2}^{k}}^{\alpha }\left( c_{J};\mu \right) &=&i\left[
\left( \frac{\alpha +k-2}{2}\right) \overline{\widehat{\Psi _{\mu }}\left(
k-1\right) }-\left( \frac{\alpha -k-2}{2}\right) \widehat{\Psi _{\mu }}%
\left( k+1\right) \right] \ .  \notag
\end{eqnarray}%
Thus $\det \left[ 
\begin{array}{c}
\mathbf{Z}_{\Omega _{1}^{k}}^{\alpha }\left( c_{J};\mu \right) \\ 
\mathbf{Z}_{\Omega _{2}^{k}}^{\alpha }\left( c_{J};\mu \right)%
\end{array}%
\right] $ is the imaginary part of $\mathbf{Z}_{\Omega _{1}^{k}}^{\alpha
}\left( c_{J};\mu \right) \ \overline{\mathbf{Z}_{\Omega _{2}^{k}}^{\alpha
}\left( c_{J};\mu \right) }$, which is $-1$ times the real part of%
\begin{eqnarray*}
&&\left\{ \left( \frac{\alpha +k-2}{2}\right) \overline{\widehat{\Psi _{\mu }%
}\left( k-1\right) }+\left( \frac{\alpha -k-2}{2}\right) \widehat{\Psi _{\mu
}}\left( k+1\right) \right\} \\
&&\times \left\{ \left( \frac{\alpha +k-2}{2}\right) \widehat{\Psi _{\mu }}%
\left( k-1\right) -\left( \frac{\alpha -k-2}{2}\right) \overline{\widehat{%
\Psi _{\mu }}\left( k+1\right) }\right\} \\
&=&\left( \frac{\alpha +k-2}{2}\right) ^{2}\left\vert \widehat{\Psi _{\mu }}%
\left( k-1\right) \right\vert ^{2}-\left( \frac{\alpha -k-2}{2}\right)
^{2}\left\vert \widehat{\Psi _{\mu }}\left( k+1\right) \right\vert ^{2} \\
&&+\func{Re}\left[ \left( \frac{\alpha +k-2}{2}\right) \left( \frac{\alpha
-k-2}{2}\right) \left( \widehat{\Psi _{\mu }}\left( k+1\right) \widehat{\Psi
_{\mu }}\left( k-1\right) -\overline{\widehat{\Psi _{\mu }}\left( k-1\right) 
}\overline{\widehat{\Psi _{\mu }}\left( k+1\right) }\right) \right] \\
&=&\left( \frac{\alpha +k-2}{2}\right) ^{2}\left\vert \widehat{\Psi _{\mu }}%
\left( k-1\right) \right\vert ^{2}-\left( \frac{\alpha -k-2}{2}\right)
^{2}\left\vert \widehat{\Psi _{\mu }}\left( k+1\right) \right\vert ^{2},
\end{eqnarray*}%
since $\widehat{\Psi _{\mu }}\left( k+1\right) \widehat{\Psi _{\mu }}\left(
k-1\right) -\overline{\widehat{\Psi _{\mu }}\left( k-1\right) }\overline{%
\widehat{\Psi _{\mu }}\left( k+1\right) }$ is pure imaginary. We conclude
that%
\begin{equation}
\det \left[ 
\begin{array}{c}
\mathbf{Z}_{\Omega _{1}^{k}}^{\alpha }\left( c_{J};\mu \right) \\ 
\mathbf{Z}_{\Omega _{2}^{k}}^{\alpha }\left( c_{J};\mu \right)%
\end{array}%
\right] =0\Longleftrightarrow \left\vert \widehat{\Psi _{\mu }}\left(
k+1\right) \right\vert =\left\vert \frac{\alpha +k-2}{\alpha -k-2}%
\right\vert \left\vert \widehat{\Psi _{\mu }}\left( k-1\right) \right\vert
,\ \ \ \ \ \text{all }k.  \label{parallel}
\end{equation}

We also have that $\det \left[ 
\begin{array}{c}
\mathbf{Z}_{\Omega _{1}^{k}}^{\alpha }\left( c_{J};\mu \right) \\ 
\mathbf{Z}_{\Omega _{1}^{\ell }}^{\alpha }\left( c_{J};\mu \right)%
\end{array}%
\right] $ is the imaginary part of $\mathbf{Z}_{\Omega _{1}^{k}}^{\alpha
}\left( c_{J};\mu \right) \ \overline{\mathbf{Z}_{\Omega _{1}^{\ell
}}^{\alpha }\left( c_{J};\mu \right) }$, i.e. the imaginary part of%
\begin{eqnarray*}
&&\left\{ \left( \frac{\alpha +k-2}{2}\right) \overline{\widehat{\Psi _{\mu }%
}\left( k-1\right) }+\left( \frac{\alpha -k-2}{2}\right) \widehat{\Psi _{\mu
}}\left( k+1\right) \right\} \\
&&\times \left\{ \left( \frac{\alpha +\ell -2}{2}\right) \widehat{\Psi _{\mu
}}\left( k+1\right) +\left( \frac{\alpha -\ell -2}{2}\right) \overline{%
\widehat{\Psi _{\mu }}\left( k+3\right) }\right\} .
\end{eqnarray*}%
If we now suppose that $\widehat{\Psi _{\mu }}\left( n\right) $ is real for
all $n$, then $\mathbf{Z}_{\Omega _{1}^{k}}^{\alpha }\left( c_{J};\mu
\right) $ is real for all $k$, and it follows that 
\begin{equation}
\det \left[ 
\begin{array}{c}
\mathbf{Z}_{\Omega _{1}^{k}}^{\alpha }\left( c_{J};\mu \right) \\ 
\mathbf{Z}_{\Omega _{1}^{\ell }}^{\alpha }\left( c_{J};\mu \right)%
\end{array}%
\right] =\func{Im}\left( \mathbf{Z}_{\Omega _{1}^{k}}^{\alpha }\left(
c_{J};\mu \right) \ \overline{\mathbf{Z}_{\Omega _{1}^{\ell }}^{\alpha
}\left( c_{J};\mu \right) }\right) =0,\ \ \ \ \ \text{all }k,\ell .
\label{parallel real}
\end{equation}%
We are now ready to construct the measure $\mu $ with an appropriate density 
$\Psi _{\mu }$. In the case $1\leq \alpha <2$ there is a choice of density
that is easy to prove positive, and we give that first. Then we give a
density for all cases $0\leq \alpha <2$, but that is much harder to prove
positive. Finally we give a particularly simple proof for the case $\alpha
=0 $.

\textbf{Construction of a density in the case }$1\leq \alpha <2$\textbf{:}

Define a density $\Psi \left( \theta \right) $ by%
\begin{equation*}
\Psi \left( \theta \right) =1+2\sum_{n=1}^{\infty }b_{n}\cos \left( 2n\theta
\right) =1+\sum_{n=1}^{\infty }b_{n}\left\{ e^{i2n\theta }+e^{-i2n\theta
}\right\} \ ,
\end{equation*}%
where%
\begin{eqnarray*}
b_{n} &=&\left\vert \frac{\alpha +\left( 2n-3\right) }{\alpha -\left(
2n+1\right) }\frac{\alpha +\left( 2n-5\right) }{\alpha -\left( 2n-1\right) }%
...\frac{\alpha +3}{\alpha -7}\frac{\alpha +1}{\alpha -5}\frac{\alpha -1}{%
\alpha -3}\right\vert \\
&=&a_{n}a_{n-1}...a_{2}a_{1},\ \ \ \ \ n\geq 1; \\
\text{with }a_{n} &=&\left\vert \frac{\alpha +\left( 2n-3\right) }{\alpha
-\left( 2n+1\right) }\right\vert =\left\vert \frac{2n-1-x}{2n-1+x}%
\right\vert \text{ if }x=2-\alpha .
\end{eqnarray*}%
Then we have%
\begin{eqnarray*}
\widehat{\Psi }\left( 2n\right) &=&b_{n}=\widehat{\Psi }\left( -2n\right) ,\
\ \ \ \ n\geq 1, \\
\widehat{\Psi }\left( k\right) &=&0\text{ if }k\text{ is odd},
\end{eqnarray*}%
and in particular that $\left\vert \widehat{\Psi }\left( k+1\right)
\right\vert =\left\vert \frac{\alpha +k-2}{\alpha -k-2}\right\vert
\left\vert \widehat{\Psi }\left( k-1\right) \right\vert $ for all $k\geq 1$.
Now choose a measure $\mu $ giving rise to the density $\Psi $. In the case $%
1\leq \alpha <2$ we have $\left\vert \frac{\alpha +k-2}{\alpha -k-2}%
\right\vert =-\frac{\alpha +k-2}{\alpha -k-2}$ for $k\geq 1$, and so from (%
\ref{altogether}) we actually obtain that $\mathbf{Z}_{\Omega
_{1}^{k}}^{\alpha }\left( c_{J};\mu \right) =0$ for all $k\geq 1$, and that $%
\mathbf{Z}_{\Omega _{2}^{k}}^{\alpha }\left( c_{J};\mu \right) $ is
imaginary for all $k\geq 1$. Thus all of the vectors $\left\{ \mathbf{Z}%
_{\Omega _{1}^{k}}^{\alpha }\left( c_{J};\mu \right) ,\mathbf{Z}_{\Omega
_{2}^{k}}^{\alpha }\left( c_{J};\mu \right) \right\} _{k=1}^{\infty }$ are
multiples of the unit vector $\left( 0,1\right) $ in the plane (it is the
failure of such a conclusion for the case $0<\alpha <1$ that forces a
different construction below).

We must now show that the density $\Psi \left( \theta \right) $ is
nonnegative. We have $\Psi \left( \theta \right) =\Phi \left( 2\theta
\right) $ where $\widehat{\Phi }\left( 0\right) =1$ and%
\begin{equation*}
\widehat{\Phi }\left( n\right) =\widehat{\Phi }\left( -n\right)
=b_{n}=a_{n}a_{n-1}...a_{2}a_{1},\ \ \ \ \ n\geq 1.
\end{equation*}%
We claim that the nonnegative sequence $\left\{ 1,b_{1},b_{2},...\right\} $
is convex for $0<x\leq 2$, and has limit $0$ as $n\rightarrow \infty $. With
this established, the density $\Phi $ is a positive sum of F\'{e}jer
kernels, and hence $\Phi \left( \theta \right) \geq 0$. Since $a_{n}=\frac{%
2n-1-x}{2n-1+x}=1-\frac{2x}{2n-1+x}$ and $\sum_{n=1}^{\infty }\frac{2x}{%
2n-1+x}=\infty $, we see that $\lim_{n\rightarrow \infty
}b_{n}=\dprod\limits_{n=1}^{\infty }\left( 1-\frac{2x}{2n-1+x}\right) =0$.
To see the convexity we note that%
\begin{eqnarray*}
b_{n+1}+b_{n-1}-2b_{n} &=&a_{n+1}a_{n}\left[ a_{n-1}...a_{2}a_{1}\right] +%
\left[ a_{n-1}...a_{2}a_{1}\right] -2a_{n}\left[ a_{n-1}...a_{2}a_{1}\right]
\\
&=&\left[ a_{n+1}a_{n}+1-2a_{n}\right] \left[ a_{n-1}...a_{2}a_{1}\right]
\end{eqnarray*}%
is positive if and only if $a_{n+1}a_{n}+1-2a_{n}$ is positive. But for $%
n\geq 2$ and $0<x\leq 2$, we have $a_{n}=\frac{2n-1-x}{2n-1+x}$ and so 
\begin{eqnarray*}
a_{n+1}a_{n}+1-2a_{n} &=&\left( a_{n+1}-2\right) a_{n}+1 \\
&=&\left( \frac{2n+1-x}{2n+1+x}-2\right) \frac{2n-1-x}{2n-1+x}+1 \\
&=&-\left( \frac{2n+1+3x}{2n+1+x}\right) \frac{2n-1-x}{2n-1+x}+1 \\
&=&\frac{\left( 2n+1+x\right) \left( 2n-1+x\right) -\left( 2n+1+3x\right)
\left( 2n-1-x\right) }{\left( 2n+1+x\right) \left( 2n-1+x\right) } \\
&=&\frac{4x^{2}+4x}{\left( 2n+1+x\right) \left( 2n-1+x\right) }>0.
\end{eqnarray*}%
This calculation is valid also when $n=1$ and $0<x\leq 1$, so it remains to
consider only the case $n=1$ and $1\leq x\leq 2$. But then we have $a_{1}=%
\frac{x-1}{1+x}$ and so 
\begin{eqnarray*}
a_{2}a_{1}+1-2a_{1} &=&\left( a_{2}-2\right) a_{1}+1 \\
&=&\left( \frac{3-x}{3+x}-2\right) \frac{x-1}{1+x}+1=\frac{6-2x}{3+x}>0.
\end{eqnarray*}

\textbf{Construction of a density in the general case }$0\leq \alpha <2$%
\textbf{:}

This time we modify the definition of our density to be 
\begin{equation*}
\widetilde{\Psi }\left( \theta \right) =1+2\sum_{n=1}^{\infty }b_{n}\cos
\left( 2n\theta \right) =1+\sum_{n=1}^{\infty }b_{n}\left\{ e^{i2n\theta
}+e^{-i2n\theta }\right\} \ ,
\end{equation*}%
where%
\begin{eqnarray*}
b_{n} &=&\frac{\alpha +\left( 2n-3\right) }{\alpha -\left( 2n+1\right) }%
\frac{\alpha +\left( 2n-5\right) }{\alpha -\left( 2n-1\right) }...\frac{%
\alpha +3}{\alpha -7}\frac{\alpha +1}{\alpha -5}\frac{\alpha -1}{\alpha -3}
\\
&=&a_{n}a_{n-1}...a_{2}a_{1},\ \ \ \ \ n\geq 1; \\
\text{where }a_{n} &=&\frac{\alpha +\left( 2n-3\right) }{\alpha -\left(
2n+1\right) }=-\frac{2n-1-x}{2n-1+x}\text{ if }x=2-\alpha .
\end{eqnarray*}%
Then we have%
\begin{eqnarray*}
\widehat{\widetilde{\Psi }}\left( 2n\right) &=&b_{n}=\widehat{\widetilde{%
\Psi }}\left( -2n\right) ,\ \ \ \ \ 1\leq n\leq N, \\
\widehat{\widetilde{\Psi }}\left( k\right) &=&0\text{ if }k\text{ is odd},
\end{eqnarray*}%
and in particular, if $\widetilde{\mu }$ is chosen to give rise to the
density $\widetilde{\Psi }$, then from (\ref{altogether}) we obtain that $%
\mathbf{Z}_{\Omega _{2}^{k}}^{\alpha }\left( c_{J};\widetilde{\mu }\right)
=0 $ for all $k\geq 1$, and that $\mathbf{Z}_{\Omega _{1}^{k}}^{\alpha
}\left( c_{J};\widetilde{\mu }\right) $ is real for all $k\geq 1$. Thus all
of the vectors $\left\{ \mathbf{Z}_{\Omega _{1}^{k}}^{\alpha }\left( c_{J};%
\widetilde{\mu }\right) ,\mathbf{Z}_{\Omega _{2}^{k}}^{\alpha }\left( c_{J};%
\widetilde{\mu }\right) \right\} _{k=1}^{\infty }$ are multiples of the unit
vector $\left( 1,0\right) $ in the plane.

Finally, we must show that the density $\widetilde{\Psi }\left( \theta
\right) $ is positive. Now%
\begin{equation*}
\widehat{\widetilde{\Psi }}\left( 2n\right) =b_{n}=a_{n}a_{n-1}...a_{2}a_{1},
\end{equation*}%
and so by B\^{o}chner's theorem (more precisely Herglotz's theorem in this
application - see e.g. Rudin \cite{Rud} for an extension to locally compact
abelian groups), it suffices to check that the following matrices are
positive semidefinite for $n\geq 2$:%
\begin{eqnarray*}
\mathbf{B}_{n} &=&\left[ 
\begin{array}{ccccc}
\widehat{\widetilde{\Psi }}\left( 0\right) & \widehat{\widetilde{\Psi }}%
\left( 2\right) & \widehat{\widetilde{\Psi }}\left( 4\right) & \cdots & 
\widehat{\widetilde{\Psi }}\left( 2n\right) \\ 
\overline{\widehat{\widetilde{\Psi }}\left( 2\right) } & \widehat{\widetilde{%
\Psi }}\left( 0\right) & \widehat{\widetilde{\Psi }}\left( 2\right) & \cdots
& \widehat{\widetilde{\Psi }}\left( 2n-2\right) \\ 
\overline{\widehat{\widetilde{\Psi }}\left( 4\right) } & \overline{\widehat{%
\widetilde{\Psi }}\left( 2\right) } & \widehat{\widetilde{\Psi }}\left(
0\right) & \cdots & \widehat{\widetilde{\Psi }}\left( 2n-4\right) \\ 
\vdots & \vdots & \vdots & \ddots & \vdots \\ 
\overline{\widehat{\widetilde{\Psi }}\left( 2n\right) } & \overline{\widehat{%
\widetilde{\Psi }}\left( 2n-2\right) } & \overline{\widehat{\widetilde{\Psi }%
}\left( 2n-4\right) } & \cdots & \widehat{\widetilde{\Psi }}\left( 0\right)%
\end{array}%
\right] \\
&=&\left[ 
\begin{array}{ccccc}
1 & a_{1} & a_{2}a_{1} & \cdots & a_{n}...a_{1} \\ 
a_{1} & 1 & a_{1} & \cdots & a_{n-1}...a_{1} \\ 
a_{2}a_{1} & a_{1} & 1 & \cdots & a_{n-2}...a_{1} \\ 
\vdots & \vdots & \vdots & \ddots & \vdots \\ 
a_{n}...a_{1} & a_{n-1}...a_{1} & a_{n-2}...a_{1} & \cdots & 1%
\end{array}%
\right] .
\end{eqnarray*}%
Since $a_{n}=-\frac{2n-1-x}{2n-1+x}$, the matrix $\mathbf{B}_{n}$ is%
\begin{equation}
\mathbf{B}_{n}\left( x\right) =\left[ 
\begin{array}{cccccc}
1 & -\frac{1-x}{1+x} & \frac{3-x}{3+x}\frac{1-x}{1+x} & \cdots & \cdots & 
\left( -1\right) ^{n+1}\frac{\left( 2n-3\right) -x}{\left( 2n-3\right) +x}%
\cdots \frac{3-x}{3+x}\frac{1-x}{1+x} \\ 
-\frac{1-x}{1+x} & 1 & -\frac{1-x}{1+x} & \cdots & \cdots & \vdots \\ 
\frac{3-x}{3+x}\frac{1-x}{1+x} & -\frac{1-x}{1+x} & 1 & \ddots &  & \vdots
\\ 
\vdots & \vdots & \ddots & \ddots & \ddots &  \\ 
\vdots & \vdots &  & \ddots & 1 & -\frac{1-x}{1+x} \\ 
\left( -1\right) ^{n+1}\frac{\left( 2n-3\right) -x}{\left( 2n-3\right) +x}%
\cdots \frac{3-x}{3+x}\frac{1-x}{1+x} & \cdots & \cdots &  & -\frac{1-x}{1+x}
& 1%
\end{array}%
\right] ,  \label{def B}
\end{equation}%
and a standard reduction in matrix theory shows that it is enough to show
that $\det \mathbf{B}_{n}\left( x\right) \geq 0$ for all $n\geq 2$.

In the appendix below, we prove that these determinants satisfy the
recursion formula%
\begin{equation}
\frac{\det \mathbf{B}_{n+1}\left( x\right) }{\det \mathbf{B}_{n}\left(
x\right) }=2^{2n}\frac{n!\left( n-1+x\right) \left( n-2+x\right) ...\left(
x\right) }{\left[ \left( 2n-1+x\right) \left( 2n-3+x\right) ...\left(
1+x\right) \right] ^{2}},\ \ \ \ \ n\geq 1.  \label{recursion}
\end{equation}%
From this recursion we immediately obtain that for $x>0$, the determinants $%
\det \mathbf{B}_{n}\left( x\right) $ and $\det \mathbf{B}_{n+1}\left(
x\right) $ have the same sign. Then since $\det \mathbf{B}_{1}\left(
x\right) =1$, induction shows that 
\begin{equation}
\det \mathbf{B}_{n}\left( x\right) >0\text{ for all }x>0,\ n\geq 1.
\label{positive det}
\end{equation}%
This completes the proof that the matrices $\mathbf{B}_{n}$ are positive
definite for all $n\geq 1$ and $x>0$, and hence that the density $\widetilde{%
\Psi }$ is positive. We note that this completes the proof of Lemma \ref{LRE}
for all $0\leq \alpha <2$.

\textbf{Construction of the density in the case }$\alpha =0$\textbf{:}

The case $\alpha =0$ corresponds to the usual singular integrals in the
plane, and for this case there is an especially simple proof of the
nonnegativity of the density $\widetilde{\Psi }$. We simply note that the
density $\widetilde{\Psi }$ is nonnegative by taking absolute values inside
the sum,%
\begin{equation*}
\widetilde{\Psi }\left( \theta \right) =1+2\sum_{n=1}^{\infty }b_{n}\cos
\left( 2n\theta \right) \geq 1-2\sum_{n=1}^{\infty }\left\vert
b_{n}\right\vert \ ,
\end{equation*}%
and then calculating that%
\begin{eqnarray*}
\left\vert b_{n}\right\vert &=&\left\vert
a_{n}a_{n-1}...a_{2}a_{1}\right\vert \\
&=&\frac{\left( 2n-3\right) }{\left( 2n+1\right) }\frac{\left( 2n-5\right) }{%
\left( 2n-1\right) }...\frac{3}{7}\frac{1}{5}\frac{1}{3} \\
&=&\frac{1}{\left( 2n+1\right) \left( 2n-1\right) }=\frac{1}{2}\left( \frac{1%
}{2n-1}-\frac{1}{2n+1}\right) ,
\end{eqnarray*}%
hence%
\begin{equation*}
\sum_{n=1}^{\infty }\left\vert b_{n}\right\vert =\sum_{n=1}^{\infty }\frac{1%
}{2}\left( \frac{1}{2n-1}-\frac{1}{2n+1}\right) =\frac{1}{2}.
\end{equation*}
\end{proof}

Now we show how to adapt the above proof to prove Corollary \ref{LRE porism}.

\begin{proof}[Proof of Corollary \protect\ref{LRE porism}]
First we note that if $\Omega $ is sufficiently smooth with vanishing
integral on the circle, then it is an absolutely convergent sum of the trig
functions $\cos n\theta $ and $\sin n\theta $ for $n\geq 1$. Thus a standard
limiting argument extends the above failure of energy reversal to any finite
vector of such $\Omega $. Now embed the measure $\widetilde{\mu }$ with
density $\widetilde{\Psi }$ constructed above into Euclidean space $\mathbb{R%
}^{n}$ via the embedding $\mathbb{R}^{2}\ni \left( x_{1},x_{2}\right)
\rightarrow \left( x_{1},x_{2},x_{3},...,x_{n}\right) \in \mathbb{R}%
^{2}\times \mathbb{R}^{n-2}$. Here we are letting the parameter $x=n-\alpha $
lie in the interval $\left( 0,n\right] $. Then the above proof shows that
for cubes $J$ with center $c_{J}\in \mathbb{R}^{2}\times \left\{ 0\right\} $%
, the gradients $\mathbf{Z}_{\Omega }^{\alpha }\left( c_{J};\widetilde{\mu }%
\right) $ of the kernels $\Omega $ have their planar projections parallel to 
$\left( 1,0\right) $, and hence all the gradients $\mathbf{Z}_{\Omega
}^{\alpha }\left( c_{J};\widetilde{\mu }\right) $ are perpendicular to the
fixed direction $\left( 0,1,0....,0\right) $ in $\mathbb{R}^{n}$. As a
consequence, reversal of energy fails in $J$ for the measure $\widetilde{\mu 
}$, and it remains only to show that the density $\widetilde{\Psi }$ is
positive. But this is implied by the positivity of $\det \mathbf{B}%
_{n}\left( x\right) $ for $x\in \left( 0,n\right] $, which follows from the
recursion (\ref{recursion}) and the fact that $\det \mathbf{B}_{1}\left(
x\right) =1>0$.
\end{proof}

\section{Appendix}

We can rewrite the recursion (\ref{recursion}) above as%
\begin{equation}
\frac{\det \mathbf{B}_{n+1}\left( x\right) }{\det \mathbf{B}_{n}\left(
x\right) }=\frac{\Omega _{n}^{n}\left( x-1\right) }{\left[ \Omega
_{n}^{n}\left( \frac{x-1}{2}\right) \right] ^{2}},\ \ \ \ \ n\geq 1,
\label{recursion''}
\end{equation}%
where for any positive integer $n$ and real number $a$ we define the
combinatorial coefficient%
\begin{equation*}
\Omega _{n}^{n}\left( a\right) \equiv \frac{\left( n+a\right) \left(
n-1+a\right) ...\left( 1+a\right) }{\left( n\right) \left( n-1\right)
...\left( 1\right) }.
\end{equation*}%
We now prove the recursion formula (\ref{recursion''}) using the well known
block determinant formula%
\begin{equation}
\det \left[ 
\begin{array}{cc}
\mathbf{B} & \mathbf{c} \\ 
\mathbf{r} & a%
\end{array}%
\right] =a\det \mathbf{B}-\mathbf{r}\left[ \limfunc{co}\mathbf{B}\right] ^{%
\limfunc{tr}}\mathbf{c}=\det \mathbf{B}\left\{ a-\mathbf{rB}^{-1}\mathbf{c}%
\right\} \ ,  \label{block det}
\end{equation}%
where $\mathbf{B}$ is an $n\times n$ matrix and $\mathbf{r}$ and $\mathbf{c}$
are $n$-dimensional row and column vectors respectively. Here $\left[ 
\limfunc{co}\mathbf{B}\right] ^{\limfunc{tr}}$ denotes the transposed
cofactor matrix of $\mathbf{B}$ and the inverse of $\mathbf{B}$ is given by $%
\mathbf{B}^{-1}=\frac{1}{\det \mathbf{B}}\left[ \limfunc{co}\mathbf{B}\right]
^{\limfunc{tr}}$. If we apply this with $\mathbf{B}=\mathbf{B}_{n}\left(
x\right) $ and $\left[ 
\begin{array}{cc}
\mathbf{B} & \mathbf{c} \\ 
\mathbf{r} & a%
\end{array}%
\right] =\mathbf{B}_{n+1}\left( x\right) $ we get%
\begin{eqnarray}
\det \mathbf{B}_{n+1}\left( x\right) &=&\det \left[ 
\begin{array}{cc}
\mathbf{B}_{n}\left( x\right) & \mathbf{c}^{n}\left( x\right) \\ 
\mathbf{r}_{n}\left( x\right) & 1%
\end{array}%
\right]  \label{block det'} \\
&=&\det \mathbf{B}_{n}\left( x\right) \left\{ 1-\mathbf{r}_{n}\left(
x\right) \mathbf{B}_{n}\left( x\right) ^{-1}\mathbf{c}^{n}\left( x\right)
\right\} ,  \notag
\end{eqnarray}%
where $\mathbf{r}_{n}\left( x\right) $ denotes the $n$-dimensional row
vector consisting of the first $n$ entries of the bottom row of $\mathbf{B}%
_{n+1}\left( x\right) $, and similarly $\mathbf{c}^{n}\left( x\right) $
denotes the $n$-dimensional column vector consisting of the first $n$
entries of the rightmost column of $\mathbf{B}_{n+1}\left( x\right) $. Note
also that $\mathbf{r}_{n}\left( x\right) $ and $\mathbf{c}^{n}\left(
x\right) $ are transposes of each other.

Motivated by computer algebra calculations, we \textbf{define} the column
vector%
\begin{equation}
\mathbf{v}^{n}\left( x\right) \equiv \left( -1\right) ^{n-1}\left[ \left(
-1\right) ^{k}\left( 
\begin{array}{c}
n \\ 
k%
\end{array}%
\right) \Gamma _{k}^{n}\left( \frac{x-1}{2}\right) \right] _{k=0}^{n-1},
\label{def v}
\end{equation}%
where%
\begin{equation*}
\Gamma _{k}^{n}\left( a\right) \equiv \frac{\Gamma \left( k+a+1\right)
\Gamma \left( n-k+a\right) }{\Gamma \left( n+a+1\right) \Gamma \left(
a\right) }=\frac{\left( k+a\right) ...\left( a\right) }{\left( n+a\right)
...\left( n-k+a\right) }.
\end{equation*}

\begin{lemma}
\label{vector}For $n\geq 1$ we have%
\begin{equation*}
\mathbf{B}_{n}\left( x\right) ^{-1}\mathbf{c}^{n}\left( x\right) =\mathbf{v}%
^{n}\left( x\right) \ .
\end{equation*}
\end{lemma}

\begin{proof}
It suffices to show the vector identity%
\begin{equation*}
\mathbf{B}_{n}\left( x\right) \mathbf{v}^{n}\left( x\right) =\mathbf{c}%
^{n}\left( x\right) ,\ \ \ \ \ n\geq 1,
\end{equation*}%
and to prove this we will use the well known fact that an $n^{th}$ order
difference of a polynomial of degree less than $n$ vanishes. More
specifically the polynomial in question will be 
\begin{equation*}
P_{n-1}\left( s\right) \equiv \frac{\Gamma \left( n-1+s\right) }{\Gamma
\left( s\right) }=\left( n-1+s\right) ...\left( 1+s\right) s.
\end{equation*}%
Indeed,%
\begin{eqnarray*}
\mathbf{v}^{n}\left( x\right) &\equiv &\left[ \left( -1\right) ^{k}\left( 
\begin{array}{c}
n \\ 
n-1-k%
\end{array}%
\right) \Gamma _{n-1-k}^{n}\left( \frac{x-1}{2}\right) \right] _{k=0}^{n-1}
\\
&=&\left[ \left( -1\right) ^{k}\left( 
\begin{array}{c}
n \\ 
k+1%
\end{array}%
\right) \frac{\Gamma \left( n-k+\frac{x-1}{2}\right) \Gamma \left( 1+k+\frac{%
x-1}{2}\right) }{\Gamma \left( n+1+\frac{x-1}{2}\right) \Gamma \left( \frac{%
x-1}{2}\right) }\right] _{k=0}^{n-1} \\
&=&\left[ \left( -1\right) ^{k-1}\left( 
\begin{array}{c}
n \\ 
k%
\end{array}%
\right) \frac{\Gamma \left( n-k+z\right) \Gamma \left( k-1+z\right) }{\Gamma
\left( n+z\right) \Gamma \left( -1+z\right) }\right] _{k=1}^{n}\ ,
\end{eqnarray*}%
where%
\begin{equation*}
z=\frac{x-1}{2}+1=\frac{x+1}{2}.
\end{equation*}%
Now we use%
\begin{eqnarray*}
&&\frac{\left( x-1\right) \left( x-3\right) \left( x-5\right) ...\left(
x-\left( 2n-1\right) \right) }{\left( x+1\right) \left( x+3\right) \left(
x+5\right) ...\left( x+\left( 2n-1\right) \right) } \\
&=&\frac{\left( \frac{x-1}{2}\right) \left( \frac{x-1}{2}-1\right) \left( 
\frac{x-1}{2}-2\right) ...\left( \frac{x-1}{2}-\left( n-1\right) \right) }{%
\left( \frac{x-1}{2}+1\right) \left( \frac{x-1}{2}+2\right) \left( \frac{x-1%
}{2}+3\right) ...\left( \frac{x-1}{2}+n\right) } \\
&=&\frac{\Gamma \left( \frac{x-1}{2}+1\right) \Gamma \left( \frac{x-1}{2}%
+1\right) }{\Gamma \left( \frac{x-1}{2}-\left( n-1\right) \right) \Gamma
\left( \frac{x-1}{2}+n+1\right) } \\
&=&\frac{\Gamma \left( z\right) ^{2}}{\Gamma \left( z-n\right) \Gamma \left(
z+n\right) },
\end{eqnarray*}%
to obtain that%
\begin{equation*}
\mathbf{B}_{n}\left( x\right) =\left[ \frac{\Gamma \left( z\right) ^{2}}{%
\Gamma \left( z-\left\vert j-i\right\vert \right) \Gamma \left( z+\left\vert
j-i\right\vert \right) }\right] _{i,j=1}^{n}
\end{equation*}%
Thus the first row of $\mathbf{B}_{n}\left( x\right) $ is%
\begin{eqnarray*}
&&\left( 
\begin{array}{ccccc}
1 & \frac{x-1}{x+1} & \frac{x-3}{x+3}\frac{x-1}{x+1} & \cdots & \frac{\left(
x-1\right) \left( x-3\right) \left( x-5\right) ...\left( x-\left(
2n-1\right) \right) }{\left( x+1\right) \left( x+3\right) \left( x+5\right)
...\left( x+\left( 2n-1\right) \right) }%
\end{array}%
\right) \\
&=&\left( 
\begin{array}{ccccc}
\frac{\Gamma \left( z\right) ^{2}}{\Gamma \left( z\right) \Gamma \left(
z\right) } & \frac{\Gamma \left( z\right) ^{2}}{\Gamma \left( z-1\right)
\Gamma \left( z+1\right) } & \frac{\Gamma \left( z\right) ^{2}}{\Gamma
\left( z-2\right) \Gamma \left( z+3\right) } & \cdots & \frac{\Gamma \left(
z\right) ^{2}}{\Gamma \left( z-\left( n-1\right) \right) \Gamma \left(
z+\left( n-1\right) \right) }%
\end{array}%
\right) \\
&=&\left[ \frac{\Gamma \left( z\right) ^{2}}{\Gamma \left( z-\left(
k-1\right) \right) \Gamma \left( z+\left( k-1\right) \right) }\right]
_{k=1}^{n}.
\end{eqnarray*}%
Thus we get%
\begin{eqnarray*}
&&\left[ \frac{\Gamma \left( z\right) ^{2}}{\Gamma \left( z-\left(
k-1\right) \right) \Gamma \left( z+\left( k-1\right) \right) }\right]
_{k=1}^{n}\cdot \mathbf{v}^{n}\left( x\right) \\
&=&-\sum_{k=1}^{n}\left( -1\right) ^{k}\left( 
\begin{array}{c}
n \\ 
k%
\end{array}%
\right) \frac{\Gamma \left( n-k+z\right) \Gamma \left( k-1+z\right) }{\Gamma
\left( n+z\right) \Gamma \left( -1+z\right) }\frac{\Gamma \left( z\right)
^{2}}{\Gamma \left( z-\left( k-1\right) \right) \Gamma \left( z+\left(
k-1\right) \right) } \\
&=&-\sum_{k=1}^{n}\left( -1\right) ^{k}\left( 
\begin{array}{c}
n \\ 
k%
\end{array}%
\right) \frac{\Gamma \left( z\right) ^{2}}{\Gamma \left( z+n\right) \Gamma
\left( z-1\right) }\frac{\Gamma \left( z-k+n\right) }{\Gamma \left(
z-k+1\right) } \\
&=&-\frac{\Gamma \left( z\right) ^{2}}{\Gamma \left( z+n\right) \Gamma
\left( z-1\right) }\sum_{k=1}^{n}\left( -1\right) ^{k}\left( 
\begin{array}{c}
n \\ 
k%
\end{array}%
\right) \left\{ \left( z-k+n-1\right) ...\left( z-k+1\right) \right\} \\
&=&-\frac{\Gamma \left( z\right) ^{2}}{\Gamma \left( z+n\right) \Gamma
\left( z-1\right) }\sum_{k=1}^{n}\left( -1\right) ^{k}\left( 
\begin{array}{c}
n \\ 
k%
\end{array}%
\right) P_{z}^{n}\left( k\right)
\end{eqnarray*}%
where $P_{z}^{n}\left( w\right) =\left( z-w+n-1\right) ...\left(
z-w+1\right) $ is a polynomial of degree $n-1$. Now recall that if $%
\bigtriangleup f\equiv f\left( 1\right) -f\left( 0\right) $ is the unit
difference operator at $0$, then 
\begin{equation*}
\bigtriangleup ^{n}f=\sum_{k=0}^{n}\left( -1\right) ^{k}\left( 
\begin{array}{c}
n \\ 
k%
\end{array}%
\right) f\left( k\right)
\end{equation*}%
Thus we have%
\begin{equation*}
\sum_{k=0}^{n}\left( -1\right) ^{k}\left( 
\begin{array}{c}
n \\ 
k%
\end{array}%
\right) P_{z}^{n}\left( k\right) =\bigtriangleup ^{n}P_{z}^{n}=0
\end{equation*}%
since $P_{z}^{n}$ has degree less than $n$, and so%
\begin{eqnarray*}
&&\left[ \frac{\Gamma \left( z\right) ^{2}}{\Gamma \left( z-\left(
k-1\right) \right) \Gamma \left( z+\left( k-1\right) \right) }\right]
_{k=1}^{n}\cdot \mathbf{v}^{n}\left( x\right) \\
&=&\frac{\Gamma \left( z\right) ^{2}}{\Gamma \left( z+n\right) \Gamma \left(
z-1\right) }\left( z+n-1\right) ...\left( z+1\right) \\
&=&\frac{\Gamma \left( z\right) ^{2}}{\Gamma \left( z+1\right) \Gamma \left(
z-1\right) }
\end{eqnarray*}%
which is the first component of $\mathbf{c}^{n}\left( x\right) $ as
required. A similar argument proves the equality of the remaining
components, and this completes the proof of Lemma \ref{vector}.
\end{proof}

\begin{lemma}
\label{vector'}For $n\geq 1$ we have%
\begin{equation*}
1-\mathbf{r}_{n}\left( x\right) \cdot \mathbf{v}_{n}\left( x\right) =\frac{%
\Omega _{n}^{n}\left( x-1\right) }{\left[ \Omega _{n}^{n}\left( \frac{x-1}{2}%
\right) \right] ^{2}}.
\end{equation*}
\end{lemma}

\begin{proof}
Again, this is an application of the fact that an $n^{th}$ order difference
of a polynomial of degree less than $n$ vanishes, but a bit more
complicated. Recall that 
\begin{equation*}
\Omega _{n}^{n}\left( a\right) \equiv \frac{\left( n+a\right) \left(
n-1+a\right) ...\left( 1+a\right) }{\left( n\right) \left( n-1\right)
...\left( 1\right) }=\frac{\Gamma \left( n+1+a\right) }{\Gamma \left(
1+a\right) n!},
\end{equation*}%
so that we have%
\begin{eqnarray*}
\frac{\Omega _{n}^{n}\left( x-1\right) }{\left[ \Omega _{n}^{n}\left( \frac{%
x-1}{2}\right) \right] ^{2}} &=&n!\frac{\left( n+x-1\right) \left(
n-1+x-1\right) ...\left( 1+x-1\right) }{\left( n+\frac{x-1}{2}\right)
^{2}\left( n-1+\frac{x-1}{2}\right) ^{2}...\left( 1+\frac{x-1}{2}\right) ^{2}%
} \\
&=&\frac{\Gamma \left( n+1\right) \Gamma \left( n+x\right) \Gamma \left( 1+%
\frac{x-1}{2}\right) ^{2}}{\Gamma \left( x\right) \Gamma \left( n+1+\frac{x-1%
}{2}\right) ^{2}}.
\end{eqnarray*}%
We also have%
\begin{eqnarray*}
\mathbf{v}^{n}\left( x\right) &\equiv &\left[ \left( -1\right) ^{k}\left( 
\begin{array}{c}
n \\ 
n-1-k%
\end{array}%
\right) \Gamma _{n-1-k}^{n}\left( \frac{x-1}{2}\right) \right] _{k=0}^{n-1}
\\
&=&\left[ \left( -1\right) ^{k}\left( 
\begin{array}{c}
n \\ 
k+1%
\end{array}%
\right) \frac{\Gamma \left( n-k+\frac{x-1}{2}\right) \Gamma \left( 1+k+\frac{%
x-1}{2}\right) }{\Gamma \left( n+1+\frac{x-1}{2}\right) \Gamma \left( \frac{%
x-1}{2}\right) }\right] _{k=0}^{n-1}\ ,
\end{eqnarray*}%
and from (\ref{def B}), we have%
\begin{eqnarray*}
\mathbf{r}_{n}\left( x\right) &=&\left( 
\begin{array}{ccccc}
\left( -1\right) ^{n}\frac{\left( 2n-1\right) -x}{\left( 2n-1\right) +x}%
\cdots \frac{3-x}{3+x}\frac{1-x}{1+x} & \cdots & \cdots & \frac{3-x}{3+x}%
\frac{1-x}{1+x} & -\frac{1-x}{1+x}%
\end{array}%
\right) \\
&=&\left[ \left( -1\right) ^{k+1}\frac{\left( 2k+1\right) -x}{\left(
2k+1\right) +x}\cdots \frac{3-x}{3+x}\frac{1-x}{1+x}\right] _{k=0}^{n-1} \\
&=&\left[ \frac{x-\left( 2k+1\right) }{\left( 2k+2\right) +x-1}\cdots \frac{%
x-3}{4+x-1}\frac{x-1}{2+x-1}\right] _{k=0}^{n-1} \\
&=&\left[ \frac{-2k+x-1}{\left( 2k+2\right) +x-1}\cdots \frac{-2+x-1}{4+x-1}%
\frac{x-1}{2+x-1}\right] _{k=0}^{n-1}\ ,
\end{eqnarray*}%
and hence dividing all factors top and bottom by $2$, we get%
\begin{eqnarray*}
\mathbf{r}_{n}\left( x\right) &=&\left[ \frac{-k+\frac{x-1}{2}}{k+1+\frac{x-1%
}{2}}\cdots \frac{-1+\frac{x-1}{2}}{2+\frac{x-1}{2}}\frac{\frac{x-1}{2}}{1+%
\frac{x-1}{2}}\right] _{k=0}^{n-1} \\
&=&\left[ \frac{\Gamma \left( 1+\frac{x-1}{2}\right) \Gamma \left( 1+\frac{%
x-1}{2}\right) }{\Gamma \left( -k+\frac{x-1}{2}\right) \Gamma \left( k+2+%
\frac{x-1}{2}\right) }\right] _{k=0}^{n-1} \\
&=&\left[ \frac{\Gamma \left( 1+\frac{x-1}{2}\right) ^{2}}{\Gamma \left( -k+%
\frac{x-1}{2}\right) \Gamma \left( k+2+\frac{x-1}{2}\right) }\right]
_{k=0}^{n-1}\ .
\end{eqnarray*}%
Thus our identity to be proved is%
\begin{eqnarray}
&&\sum_{k=0}^{n-1}\left( -1\right) ^{k}\left( 
\begin{array}{c}
n \\ 
k+1%
\end{array}%
\right) \frac{\Gamma \left( 1+\frac{x-1}{2}\right) ^{2}}{\Gamma \left( -k+%
\frac{x-1}{2}\right) \Gamma \left( k+2+\frac{x-1}{2}\right) }  \label{TBP} \\
&&\ \ \ \ \ \ \ \ \ \ \ \ \ \ \ \ \ \ \ \ \ \ \ \ \ \ \ \ \ \ \ \ \ \ \ \ \
\ \ \ \times \frac{\Gamma \left( n-k+\frac{x-1}{2}\right) \Gamma \left( k+1+%
\frac{x-1}{2}\right) }{\Gamma \left( n+1+\frac{x-1}{2}\right) \Gamma \left( 
\frac{x-1}{2}\right) }  \notag \\
&=&1-\frac{\Gamma \left( n+1\right) \Gamma \left( n+x\right) \Gamma \left( 1+%
\frac{x-1}{2}\right) ^{2}}{\Gamma \left( x\right) \Gamma \left( n+1+\frac{x-1%
}{2}\right) ^{2}}.  \notag
\end{eqnarray}%
If we set $z=1+\frac{x-1}{2}$ then this identity becomes%
\begin{eqnarray*}
&&\sum_{k=0}^{n-1}\left( -1\right) ^{k}\left( 
\begin{array}{c}
n \\ 
k+1%
\end{array}%
\right) \frac{\Gamma \left( z\right) ^{2}}{\Gamma \left( -k-1+z\right)
\Gamma \left( k+1+z\right) }\frac{\Gamma \left( n-k-1+z\right) \Gamma \left(
k+z\right) }{\Gamma \left( n+z\right) \Gamma \left( -1+z\right) } \\
&=&1-\frac{\Gamma \left( n+1\right) \Gamma \left( n-1+2z\right) \Gamma
\left( z\right) ^{2}}{\Gamma \left( -1+2z\right) \Gamma \left( n+z\right)
^{2}},
\end{eqnarray*}%
and if we replace $k$ by $k-1$ we get%
\begin{eqnarray*}
&&\sum_{k=1}^{n}\left( -1\right) ^{k-1}\left( 
\begin{array}{c}
n \\ 
k%
\end{array}%
\right) \frac{\Gamma \left( z\right) ^{2}}{\Gamma \left( -k+z\right) \Gamma
\left( k+z\right) }\frac{\Gamma \left( n-k+z\right) \Gamma \left(
k-1+z\right) }{\Gamma \left( n+z\right) \Gamma \left( -1+z\right) } \\
&=&1-\frac{\Gamma \left( n+1\right) \Gamma \left( n-1+2z\right) \Gamma
\left( z\right) ^{2}}{\Gamma \left( -1+2z\right) \Gamma \left( n+z\right)
^{2}}.
\end{eqnarray*}%
Note that the term $k=0$ in the sum on the left would be $-1$, so that we
can subtract $1$ from both sides, and then multiply by $-1$ to get%
\begin{eqnarray*}
&&\sum_{k=0}^{n}\left( -1\right) ^{k}\left( 
\begin{array}{c}
n \\ 
k%
\end{array}%
\right) \frac{\Gamma \left( z\right) ^{2}}{\Gamma \left( -k+z\right) \Gamma
\left( k+z\right) }\frac{\Gamma \left( n-k+z\right) \Gamma \left(
k-1+z\right) }{\Gamma \left( n+z\right) \Gamma \left( -1+z\right) } \\
&=&\frac{\Gamma \left( n+1\right) \Gamma \left( n-1+2z\right) \Gamma \left(
z\right) ^{2}}{\Gamma \left( -1+2z\right) \Gamma \left( n+z\right) ^{2}},
\end{eqnarray*}%
which is equivalent to%
\begin{equation}
\sum_{k=0}^{n}\left( -1\right) ^{k}\left( 
\begin{array}{c}
n \\ 
k%
\end{array}%
\right) \frac{\Gamma \left( z+n-k\right) \Gamma \left( z+k-1\right) }{\Gamma
\left( z-k\right) \Gamma \left( z+k\right) }=\frac{\Gamma \left( n+1\right)
\Gamma \left( z-1\right) \Gamma \left( 2z+n-1\right) }{\Gamma \left(
z+n\right) \Gamma \left( 2z-1\right) }.  \label{TBP'}
\end{equation}%
We now use 
\begin{equation*}
\frac{\Gamma \left( s+m+1\right) }{\Gamma \left( s\right) }=\left(
s+m\right) \left( s+m-1\right) ...\left( s+1\right) s
\end{equation*}%
to rewrite (\ref{TBP'}) as%
\begin{equation}
\sum_{k=0}^{n}\left( -1\right) ^{k}\left( 
\begin{array}{c}
n \\ 
k%
\end{array}%
\right) \frac{\left( z+n-k-1\right) ...\left( z-k\right) }{\left(
z+k-1\right) }=n!\frac{\left( 2z+n-2\right) ...\left( 2z\right) \left(
2z-1\right) }{\left( z+n-1\right) ...\left( z\right) \left( z-1\right) }.
\label{TBP''}
\end{equation}

Denote the left and right hand sides of (\ref{TBP''}) by $LHS_{n}\left(
z\right) $ and $RHS_{n}\left( z\right) $ respectively. Then the left hand
side $LHS_{n}\left( z\right) $ of (\ref{TBP''}) is%
\begin{eqnarray*}
LHS_{n}\left( z\right) &=&\sum_{k=0}^{n}\left( -1\right) ^{k}\left( 
\begin{array}{c}
n \\ 
k%
\end{array}%
\right) \frac{\left( z+n-k-1\right) ...\left( z+1-k\right) \left( \left[
z+k-1\right] -\left[ 2k-1\right] \right) }{\Gamma \left( z+k-1\right) } \\
&=&\sum_{k=0}^{n}\left( -1\right) ^{k}\left( 
\begin{array}{c}
n \\ 
k%
\end{array}%
\right) \left( z+n-k-1\right) ...\left( z+1-k\right) \\
&&+\sum_{k=0}^{n}\left( -1\right) ^{k+1}\left( 
\begin{array}{c}
n \\ 
k%
\end{array}%
\right) \frac{\left( z+n-k-1\right) ...\left( z+1-k\right) }{\left(
z+k-1\right) }\left( 2k-1\right) ,
\end{eqnarray*}%
where the first sum on the right hand side above vanishes since it is an $%
n^{th}$ order difference of the polynomial 
\begin{equation*}
P\left( w\right) \equiv \left( z+n-w-1\right) ...\left( z+1-w\right)
\end{equation*}%
of degree $n-1$. Thus we have%
\begin{eqnarray*}
LHS_{n}\left( z\right) &=&\sum_{k=0}^{n}\left( -1\right) ^{k+1}\left( 
\begin{array}{c}
n \\ 
k%
\end{array}%
\right) \frac{\left( z+n-k-1\right) ...\left( z+2-k\right) \left( \left[
z+k-1\right] -\left[ 2k-2\right] \right) }{\left( z+k-1\right) }\left(
2k-1\right) \\
&=&\sum_{k=0}^{n}\left( -1\right) ^{k+1}\left( 
\begin{array}{c}
n \\ 
k%
\end{array}%
\right) \left( z+n-k-1\right) ...\left( z+2-k\right) \left( 2k-1\right) \\
&&+\sum_{k=0}^{n}\left( -1\right) ^{k+2}\left( 
\begin{array}{c}
n \\ 
k%
\end{array}%
\right) \frac{\left( z+n-k-1\right) ...\left( z+2-k\right) }{\left(
z+k-1\right) }\left( 2k-2\right) \left( 2k-1\right) ,
\end{eqnarray*}%
where the first sum on the right hand side above vanishes since it is an $%
n^{th}$ order difference of the polynomial%
\begin{equation*}
P\left( w\right) \equiv \left( z+n-w-1\right) ...\left( z+2-w\right) \left(
2w-1\right)
\end{equation*}%
of degree $n-1$. Continuing in this way we get%
\begin{equation*}
LHS_{n}\left( z\right) =\sum_{k=0}^{n}\left( -1\right) ^{k+n}\left( 
\begin{array}{c}
n \\ 
k%
\end{array}%
\right) \frac{1}{\left( z+k-1\right) }\left( 2k-n\right) ...\left(
2k-2\right) \left( 2k-1\right) .
\end{equation*}

Now the right hand side $RHS_{n}\left( z\right) $ of (\ref{TBP''}) is a
quotient of a polynomial of degree $n$ by a polynomial of degree $n+1$, and
so has a partial fraction decomposition of the form%
\begin{equation*}
RHS_{n}\left( z\right) =n!\frac{\left( 2z+n-2\right) ...\left( 2z\right)
\left( 2z-1\right) }{\left( z+n-1\right) ...\left( z\right) \left(
z-1\right) }=\sum_{k=0}^{n}\frac{A_{k}}{z+k-1},
\end{equation*}%
for uniquely determined coefficients $A_{0},...A_{n}$. So the proof of (\ref%
{TBP''}) has been reduced to proving the identity, 
\begin{equation}
A_{k}=\left( -1\right) ^{k+n}\left( 
\begin{array}{c}
n \\ 
k%
\end{array}%
\right) \left( 2k-n\right) ...\left( 2k-2\right) \left( 2k-1\right) .
\label{Ak equals}
\end{equation}%
Now $A_{k}$ is the residue of the meromorphic function $RHS_{n}\left(
z\right) $ at $z=-\left( k-1\right) $, hence using the notation $\widehat{%
\left( z+k-1\right) }$ to indicate that the factor $\left( z+k-1\right) $ is 
\emph{missing}, we get 
\begin{eqnarray*}
A_{k} &=&res\left( RHS_{n}\left( z\right) ;-\left( k-1\right) \right) \\
&=&n!\frac{\left( 2z+n-2\right) ...\left( 2z\right) \left( 2z-1\right) }{%
\left( z+n-1\right) ...\left( z+k\right) \widehat{\left( z+k-1\right) }%
\left( z+k-2\right) ...\left( z\right) \left( z-1\right) }\mid _{z=-\left(
k-1\right) } \\
&=&n!\frac{\left( 2\left[ 1-k\right] +n-2\right) ...\left( 2\left[ 1-k\right]
\right) \left( 2\left[ 1-k\right] -1\right) }{\left( \left[ 1-k\right]
+n-1\right) ...\left( \left[ 1-k\right] +k\right) \widehat{\left( \left[ 1-k%
\right] +k-1\right) }\left( \left[ 1-k\right] +k-2\right) ...\left( \left[
1-k\right] \right) \left( \left[ 1-k\right] -1\right) } \\
&=&n!\frac{\left( -1\right) ^{n}\left( 2k-n\right) ...\left( 2k-2\right)
\left( 2k-1\right) }{\left( n-k\right) ...\left( 1\right) \widehat{\left(
0\right) }\left( -1\right) ^{k}\left( 1\right) ...\left( k-1\right) \left(
k\right) } \\
&=&\left( -1\right) ^{n-k}\frac{n!}{\left( n-k\right) !k!}\left( 2k-n\right)
...\left( 2k-2\right) \left( 2k-1\right) ,
\end{eqnarray*}%
which proves (\ref{Ak equals}). This completes the proof of Lemma \ref%
{vector'}.
\end{proof}

The proof of our claimed recursion (\ref{recursion''}) is now completed by
combining Lemmas \ref{vector} and \ref{vector'} with (\ref{block det'}).

\end{document}